\documentclass[12pt]{amsart}
\usepackage{bbm}
\usepackage{mathabx}
\usepackage{tensor}
\usepackage{enumerate}
\usepackage{amssymb}
\usepackage{mathtools}
\usepackage{amsmath}
\usepackage{amsfonts}
\usepackage{amsthm}
\usepackage{stmaryrd}
\usepackage{paralist}
\usepackage{indentfirst}

\pdfoutput=1
\usepackage[svgnames]{xcolor}
\usepackage{tikz}

\def\nudge{.5}

\tikzset{axis/.style={ultra thick, Black!75!black, -latex, shorten <=-\nudge cm, shorten >=-2*\nudge cm}}
\tikzset{line/.style={thick,Red}}

\usepackage[all]{xy}
\usepackage{mathrsfs}
\usepackage{graphicx}
\usepackage{color}
\usepackage{tikz-cd}
\usepackage{tkz-euclide}
\usepackage{tikz}
\usetikzlibrary{matrix}
\usepackage{enumitem}
\usepackage{amsmath,amscd}
\usepackage{pifont}
\usetikzlibrary{arrows}
\usepackage[all]{xy}
\usepackage{graphicx}
\usepackage{placeins}
\usepackage{xspace}
\usepackage{blindtext}
\usepackage{geometry}

\usetikzlibrary{hobby}
\usepackage{relsize}

\usetikzlibrary{calc,intersections,through,backgrounds}

\def\rank{\mathrm{rank}}

\def\tv{\widetilde{\ve}}
\newtheorem{thm}{Theorem}[section]
\newtheorem{theorem}[thm]{Theorem}

\newtheorem{lemma}[thm]{Lemma}
\newtheorem{corollary}[thm]{Corollary}
\newtheorem{conjecture}[thm]{Conjecture}
\newtheorem{claim}{Claim}

\newtheorem{proposition}[thm]{Proposition}
\theoremstyle{definition}

\newtheorem{definition}[thm]{Definition}

\newtheorem{remark}[thm]{Remark}
\newtheorem{convention}[thm]{Convention}
\newtheorem{notation}[thm]{Notation}

\newtheorem{notation-definition}[thm]{Notation-Definition}

\numberwithin{equation}{section}

\def\OD{\overline \Delta}

\def\calH{\mathcal{H}}
\def\calK{\mathcal{K}}

\def\calO{\mathcal{O}}

\def\CC{\mathbb{C}}
\def\FF{\mathbb{F}}

\def\HH{\mathbb{H}}
\def\LL{\mathbb{L}}

\def\QQ{\mathbb{Q}}

\def\ZZ{\mathbb{Z}}

\def\bfm{\mathbf{m}}

\def\rmK{\mathrm{K}}

\def\jing{\overunderset{u}{i=1}{\mathlarger{\#}}}

\def\Or{\overline r}
\def\Orp{\overline r_p}
\def\adic{\textrm{-adic}}
\def\Bo{B^{\mathrm{op}}}

\def\GQp{\Gal_{\QQ_p}}
\def\rI{\Or|_{I_{\QQ_p}}}
\def\rIp{\Orp|_{I_{\QQ_p}}}
\def\ve{\varepsilon}
\def\us{\underline{s}}

\def\nS{nS}

\def\vl{v_p(\ell)}
\def\v{v_p}

\def\nS{\mathrm{nS}}

\def\NP{\mathrm{NP}}

\def\tH{\widetilde{\mathrm{H}}}

\DeclareMathOperator{\GL}{GL}

\DeclareMathOperator{\Hom}{Hom}

\DeclareMathOperator{\Gal}{Gal}

\newcommand{\ur}{\mathrm{ur}}

\newcommand{\Iw}{\mathrm{Iw}}

\newcommand{\new}{\mathrm{new}}

\newcommand{\Dig}{\mathrm{Dig}}

\begin{document}
\title{The slope-invariant of local ghost series under direct sum}

\author{Rufei Ren}
\address{Department of Mathematics 
	Fudan University}
\email{rufeir@fudan.edu.cn}
\date{\today}
	\begin{abstract}
The	ghost conjecture is first provided by Bergdall and Pollack in \cite{BP1,BP2} to study the $U_p$-slopes of spaces of modular forms, which, so far, has already brought plenty of important results.  
		The local version of this conjecture under genericity condition has been solved by Liu-Truong-Xiao-Zhao in \cite{xiao, xiao1}. 
		
		In the current  paper, we prove a necessary and sufficient condition for 
		a sequence of local ghost series to satisfy that their product has the same Newton polygon to the ghost series build from the direct sum of their associated modules (see Theorem~\ref{keythm}). That answers a common question asked in both \cite{BP2,xiao}.  
	
\end{abstract}

\subjclass[2010]{}
\keywords{Modular forms, Ghost conjecture}
\maketitle
\setcounter{tocdepth}{1}
\tableofcontents

\section{Introduction}

%

We fix a prime $p\geq 7$ and an integer $N$ that is coprime to $p$. Let $\rmK_p:=\GL_2(\ZZ_p)$.  We normalize the $p$-adic valuation of elements in $\overline\QQ_p$ so that $v_{p}(p)=1$.
For any even integer $k\geq 2$, we recall that  \begin{multline*}
	S_k(\Gamma_{0}(Np)):=
	\Big\{	f:\calH\to \CC \textrm{~is holomophic and vanishes at all cusps}\;\Big|\;
	\\
	f\left(\tfrac{\alpha z+\beta}{\gamma z+\delta}\right)	=(\gamma z+\delta)^k f(z) \textrm{~for~} \begin{bsmallmatrix}\alpha & \beta \\ \gamma &\delta\end{bsmallmatrix} \in \Gamma_{0}(Np)\Big\}
\end{multline*}
is the space of cusp forms of weight $k$ and level $\Gamma_0(Np)$, where $\calH$ is the upper half plane of $\CC$. There is a natural $U_p$-operator on $S_k(\Gamma_{0}(Np))$, whose slopes, i.e., the $p$-adic valuation of the eigenvalues of this $U_p$-operator acting on 
an eigenform basis of $S_k(\Gamma_{0}(Np))$, relates to many famous conjectures. For example, as mentioned in \cite{xiao}, 
\begin{enumerate}
	\item a folklore conjecture of Breuil--Buzzard--Emerton on the slopes of Kisin's crystabelian deformation spaces, 
	\item Gouvêa's $\left\lfloor\frac{k-1}{p+1}\right\rfloor$ conjecture,
	\item the finiteness of irreducible components of the eigencurve,
	\item Gouvêa--Mazur conjecture, 
	\item Gouvêa's conjecture on slope distributions, and
	\item  a refined version of Coleman's spectral halo conjecture.
\end{enumerate} 
We refer \cite{Bu,BC,Cl,Gou,GM92,Kis,Lo07,R}, etc., for more details, and \cite{xiao} for a summary. 
In recent breakthrough work of Bergdall and Pollack \cite{BP1,BP2}, they unified all
historically important conjectures regarding slopes into one conjecture: the ghost conjecture, which has classical and local two versions.

The latter one is stated in \cite[Conjecture~7.1]{BP2} and proved by Liu, Truong, Xiao and Zhao in \cite{xiao, xiao1} under some genericity hypothesis (see Theorem~\ref{local ghost}). Right after this conjecture in \cite{BP2}, Bergdall and Pollack asked a natural question, which they believe to be true. Here, we rephrase this question as the following conjecture.
\begin{conjecture}[\cite{BP2}]\label{conj:1}
	Given two $\calO\llbracket \rmK_p\rrbracket$-projective augmented modules $\tH_1$ and $\tH_2$ of related type (see Definition~\ref{re::2}), let $G_1(w,-)$, $G_2(w,-)$ and $G_3(w,-)$ denote, respectively, the corresponding ghost series of $\tH_1$,  $\tH_2$ and $\tH_1\oplus\tH_2$.  Then the Newton polygons of $(G_1\cdot G_2)(w_\star,-)$ and  $G_3(w_\star,-)$ are the same for any $w_\star\in \bfm_{\CC_p}$, where $\bfm_{\CC_p}$ is the maximal ideal of  $\calO_{\CC_p}$.
\end{conjecture}
The same question is also mentioned in \cite[Remark~2.32]{xiao}. In the current paper, we answer this question by proving a necessary and sufficient condition for finitely many $\calO\llbracket \rmK_p\rrbracket$-projective augmented modules of $\ve$-related types to satisfy the property described in Conjecture~\ref{conj:1}   (see Theorem~\ref{keythm}). 
%
To present our result better, we first revisit the local ghost conjecture.
 
Let $E$ be a finite field extension of $\QQ_p$ that contains $\sqrt p$, and denote by  $\calO$ and $\FF$ its
ring of integers and residue field, respectively. Let $\Delta \cong(\mathbb{Z} / p \mathbb{Z})^{\times}$ be the torsion subgroup of $\mathbb{Z}_{p}^{\times}$, and $\omega: \Delta \rightarrow \mathbb{Z}_{p}^{\times}$ be the Teichmüller lift. For an element $\alpha \in \mathbb{Z}_{p}^{\times}$, we use $\overline{\alpha} \in \Delta$ to denote its reduction modulo $p$.

%
%
We call a pair of integers $(a,b)$ with $0\leq a\leq p-2$, $0\leq b\leq p-2$ is \emph{related} to a character $\ve: \Delta^{2} \rightarrow \calO^{\times}$ if $\ve(\overline{\alpha}, \overline{\alpha})=\omega(\overline{\alpha})^{a+2b}$ for all $\overline{\alpha} \in \Delta$. 
We call a representation  
$\overline{r}: \Gal_{\QQ}:=\Gal\left(\overline{\QQ} / \QQ\right) \rightarrow \GL_{2}(\FF)$ \emph{hybrid} if $\Or$ is absolutely irreducible, and 
$\Orp:=\Or|_{\Gal_{\QQ_{p}}}$ is reducible. In this case, 
 there is a unique pair of integers $(a,b)$ such that 
\begin{equation*}
	\overline{r}|_{\mathrm{I}_{\QQ_p}} \simeq\left(\begin{array}{cc}
		\omega_{1}^{a+b+1} & *\\
		0 & \omega_{1}^{b}
	\end{array}\right) \quad \text { for some~} 0 \leq a \leq p-2 \text { and } 0 \leq b \leq p-2,	
\end{equation*}
where  $I_{\QQ_p}$ and 	$\omega_{1}$ are the inertial group and the first fundamental character of $\Gal_{\QQ_p}$, respectively. 
We call
$\Or$ \begin{itemize}
	\item \emph{of type $(a,b)$}; 
	\item \emph{generic}, if $1\leq a\leq p-4$; and
	\item \emph{related to $\ve$}, if  
	$(a,b)$ and $\ve$ are related.
\end{itemize} 
Let $\Or$ be an $\ve$-related hybrid  representation of $\Gal_\QQ$. For any $\mathrm{K}_{p}$-projective arithmetic module $\tH$ of type $\overline{r}$ (see Definition~\ref{re::1}), we define later in Notation-Definition~\ref{re:7} the abstract $p$-adic space $S_{\widetilde{H}}^{(\ve)}$, and denote by $C_{\widetilde{\mathrm{H}}}^{(\varepsilon)}(w, t) \in \mathcal{O} \llbracket w, t \rrbracket$ the characteristic power series of the $U_p$-operator on  $S_{\tH}^{(\ve)}$. 
 For a power series $f(t)=\sum\limits_{n=0}^\infty a_nt^n\in \CC_p[\![t]\!]$, the lower convex hull of points $\{\left(n, v_{p}\left(a_n\right)\right)\}_{n\geq 0}$ is called the \emph{Newton polygon of $f$}, and denoted by $\mathrm{NP}(f)$.
For two Newton polygons $A$ and $B,$ we denote by $A \# B$ the Newton polygon whose slopes (with multiplicity) is the disjoint union of those of $A$ and $B$.

\begin{theorem}[Local ghost conjecture, {\cite[Theorem~8.7]{xiao1}}]\label{local ghost}
Let $\ve: \Delta^{2} \rightarrow \calO^{\times}$ be a character; $\Or$ be an $\ve$-related hybrid  representation of type $(a, b)$ with $a\in\{2, \ldots, p-5\}$, i.e., $(a,b)$ is very generic; and
		 $\widetilde{\mathrm{H}}$ be a $\mathrm{K}_{p}$-projective arithmetic module of type $\overline{r}$ with multiplicity $m(\widetilde{\mathrm{H}})$.
 Then for every $w_{\star} \in \mathfrak{m}_{\mathbb{C}_{p}}$,
	\begin{enumerate}
		\item 
		if $\rIp$ is nonsplit, then the Newton polygon $\mathrm{NP}\left(C_{\widetilde{\mathrm{H}}}^{(\varepsilon)}\left(w_{\star},-\right)\right)$ is the same as the Newton polygon $\mathrm{NP}\left(G_{(a,b)}^{(\varepsilon)}\left(w_{\star},-\right)\right)$, stretched in both $x$-and $y$-directions by $m(\widetilde{\mathrm{H}})$, where $G_{(a,b)}^{(\varepsilon)}\left(w_{\star},-\right)$ is defined in Definition~\ref{re:3};
		\item if $\rIp$ is split, the Newton polygon $\mathrm{NP}\left(C_{\widetilde{\mathrm{H}}}^{(\varepsilon)}\left(w_{\star},-\right)\right)$ is the same as the Newton polygon $\operatorname{NP}\left(G_{(a,b)}^{(\varepsilon)}\left(w_{\star},-\right)\right) \# \mathrm{NP}\left(G_{(p-3-a, a+b+1)}^{(\varepsilon)}\left(w_{\star},-\right)\right)$, stretched in both $x$- and $y$-directions by $m(\widetilde{\mathrm{H}})$.
	\end{enumerate}
\end{theorem}

We state our main result as follows, whose proof is given at the end of \S\ref{section4.2}.
\begin{theorem}\label{keythm}
	For a character $\ve: \Delta^2 \to  \calO^\times$, and a sequence of generic $\ve$-related hybrid  representations $\Or_1,\dots, \Or_u$ of types $(a_i, b_i)$ for $i=1,\dots,u$, respectively. 
	Then for any $\calO\llbracket \rmK_p\rrbracket$-projective augmented modules $\tH_i$ of types $\Or_i$  (for $1\leq i\leq u$). The following statements are equivalent.
	\begin{enumerate}
		\item For every $w_\star\in \bfm_{\CC_p}$, we have $$\NP\left(G^{(\ve)}_{\bigoplus\limits_{i=1}^u \tH_i}(w_\star,-)\right)= \overunderset{u}{i=1}{\mathlarger{\#}}\NP\left(G^{(\ve)}_{\tH_i}(w_\star,-)\right),$$
		where $G^{(\ve)}_{\bigoplus\limits_{i=1}^u \tH_i}(w_\star,-)$ and $G^{(\ve)}_{\tH_i}(w_\star,-)$ (for $i=1,\dots, u$) are defined in  Definition~\ref{re:3} which are the ghost series associated to $\bigoplus\limits_{i=1}^u \tH_i$ and $\tH_i$, respectively.
		
		\item For every distinct $i, j\in \{0,1,\dots, p-2\}$, we have $$(a_j,b_j)=(a_i,b_i)\quad\textrm{or}\quad (a_j,b_j)=(p-3-a_i, \{a_i+b_i-1\}),$$
		where $\{a_i+b_i-1\}\in \{0, \ldots, p-2\}$ is the residue of $a_i+b_i-1$ modulo $p-1$.
	\end{enumerate}	
\end{theorem}
\begin{remark}
	This theorem shows that Conjecture~\ref{conj:1} is false in general.
\end{remark}


	\subsection*{Acknowledgment}
The author would like to express his deepest appreciation to Liang Xiao for his massive help. This paper would not be done without his inspiration. We also thank Bin Zhao for related
interesting discussions.

\section{Local ghost conjecture}

 Let $E$ be a finite extension of $\QQ_p$ that contains $\sqrt{p}$, whose  
ring of integers and residue field are denoted by $\calO$ and $\FF$, respectively. We normalize the $p$-adic valuation of elements in $\overline\QQ_p$ so that $v_{p}(p)=1$. 
Let  $\omega:\FF_p^\times\to \calO^\times$ be the Teichmüller lift, and $\rmK_p:=\GL_2(\ZZ_p)$. 
Let $\Delta:=\ZZ/(p-1)\ZZ$, and $\bfm_{\CC_p}$ be the maximal ideal of $\calO_{\CC_p}$.
For any integer $n$, we denote by $\{n\}$ the residue of $n$ in $\{0,\dots.p-2\}$ modulo $p-1$.

\subsection{The $\calO\llbracket \rmK_p\rrbracket$-projective augmented module $\tH$ of type $\Or$}
Now we introduce some notations and definitions from \cite{xiao1}. 
\begin{notation-definition}\label{re::2}\noindent
	\begin{enumerate}
		\item 
		We call a representation  
		$\overline{r}: \Gal_{\QQ}:=\Gal\left(\overline{\QQ} / \QQ\right) \rightarrow \GL_{2}(\FF)$ \emph{hybrid} if $\Or$ is absolutely irreducible, and 
		$\Orp:=\Or|_{\Gal_{\QQ_p}}$ is reducible, i.e.,
		$$
		\overline{r}_{p}=\begin{bmatrix}
			\operatorname{ur}\left(\overline{\alpha}_{1}\right) \omega_{1}^{a+b+1} & * \\
			0 & \operatorname{ur}\left(\overline{\alpha}_{2}\right) \omega_{1}^{b}
		\end{bmatrix}: \mathrm{Gal}_{\mathbb{Q}_{p}} \rightarrow \mathrm{GL}_{2}(\mathbb{F})
		$$
		with $a \in\{0, \ldots, p-2\}, b \in\{0, \ldots, p-2\}$, and $\overline{\alpha}_{1}, \overline{\alpha}_{2} \in \mathbb{F}^{\times}$, where 
		\begin{itemize}
			\item	
			$\operatorname{ur}(\alpha)$ is the one-dimensional unramified representation of $\GQp$ sending the geometric Frobenius element to $\alpha\in \FF$; and
			\item  $\omega_{1}: \GQp\rightarrow \operatorname{Gal}\left(\mathbb{Q}_{p}\left(\mu_{p}\right) / \mathbb{Q}_{p}\right) \cong \mathbb{F}_{p}^{\times}$ denote the first fundamental character of $\mathrm{Gal}_{\mathbb{Q}_{p}}$. Here, $\mu_p$ is any primitive $p$-th root of unity. 
		\end{itemize}
		\item We call $\Or$ generic if $1\leq a\leq p-4$.
		\item
		We say $\Orp$ is split if $*=0$ and nonsplit if $* \neq 0$. The condition on $a$ ensures that there is a unique such nontrivial extension when $\overline{r}_{p}$ is nonsplit, because $\mathrm{H}^{1}\left(\mathrm{Gal}_{\mathbb{Q}_{p}}, \operatorname{ur}\left(\overline{\alpha}_{2}^{-1} \overline{\alpha}_{1}\right) \omega^{a+1}\right)$ is one-dimensional.
		We write $\rIp: \mathrm{I}_{\mathbb{Q}_{p}} \rightarrow \mathrm{GL}_{2}(\mathbb{F})$ for the corresponding residual inertia representation:
		\begin{itemize}
			\item (nonsplit case) $\rI=\begin{bmatrix}\omega_{1}^{a+b+1} & * \neq 0 \\ 0 & \omega_{1}^{b}\end{bmatrix}$, where $*$ is the unique nontrivial extension (up to isomorphism) in the class $\mathrm{H}^{1}\left(\mathrm{I}_{\mathbb{Q}_{p}}, \omega_{1}^{a+1}\right)^{\mathrm{Gal}_{\mathbb{F}_p}}=\mathrm{H}^{1}\left(\mathrm{Gal}_{\mathbb{Q}_{p}}, \omega_{1}^{a+1}\right)$; and
			\item (split case) $\rI=\omega_{1}^{a+b+1} \oplus \omega_{1}^{b}$.
		\end{itemize}
		\item	
		Let $R_{\overline{r}_{p}}^{\square}$ denote the framed deformation ring of $\overline{r}_{p}$ parametrizing deformations of $\overline{r}_{p}$ into matrix representations of $\mathrm{Gal}_{\mathbb{Q}_{p}}$ with coefficients in noetherian complete local $\mathcal{O}$-algebras. 
	\end{enumerate}
	
	For a character $\eta:\Delta\to \calO^\times$ and an integer $k\geq 2$, we define the following.
	\begin{enumerate}		
		\item[(5)] 	An \emph{$\calO\llbracket \rmK_p\rrbracket$-projective augmented module} $\widetilde{\mathrm{H}}$ is a finite projective right $\mathcal{O} \llbracket\rmK_p \rrbracket$-module  together with an action of $\GL_{2}\left(\mathbb{Q}_{p}\right)$ extended from the $\rmK_p$-action.
		For such a module $\tH$, we put $$\mathrm{S}_{\tH,k}^{\mathrm{ur}}(\eta):=\Hom_{\calO[K_p]}\left(\widetilde{\mathrm{H}}, \mathcal{O}[z]^{\leq k-2} \otimes \eta \circ\det\right).$$
		Note that every $\begin{bmatrix}\alpha& \beta \\ \gamma& \delta\end{bmatrix}\in \mathrm{M}_{2}\left(\mathbb{Z}_{p}\right)$ 
		with $\alpha \delta-\beta \gamma=p^{r} d\neq 0$ acts on $h\in \mathcal{O}[z]^{\leq k-2} \otimes \eta \circ\det$ by
		$$
		\left.h\right|_{\begin{bsmallmatrix}\alpha& \beta \\ \gamma& \delta\end{bsmallmatrix}}(z)=\eta(\overline{d}) \cdot(\gamma z+\delta)^{k-2} h\left(\frac{\alpha z+\beta}{\gamma z+\delta}\right) .
		$$
		This defines the operators $T_{p}$ and $S_p$ on $\phi\in\mathrm{S}_{\tH,k}^{\mathrm{ur}}(\eta)$ as follows:
		\begin{itemize}
			\item ($T_p$-operator)
			Taking a coset decomposition $\mathrm{K}_{p}\begin{bmatrix}p^{-1} & 0 \\ 0&1\end{bmatrix}\mathrm{K}_{p}=$ $\coprod_{j=0}^{p} u_{j} \mathrm{~K}_{p}$ (e.g. $u_{j}=\begin{bmatrix}
				p^{-1} & 0 \\ j & 1\end{bmatrix}$ 
			and $u_{j}^{-1}=\begin{bmatrix}p & 0 \\ j p & 1\end{bmatrix}$ for $j=0, \ldots, p-1$, and $u_{p}\begin{bmatrix}
				p^{-1} & 0 \\ 0 & 1\end{bmatrix}$ and $\left.u_{p}^{-1}=
			\begin{bmatrix}
				p & 0 \\ 0 & 1\end{bmatrix}\right)$, for any $x \in \widetilde{\mathrm{H}}$, we have
			$$
			T_{p}(\varphi)(x):=\sum\limits_{j=0}^{p} \varphi\left(x u_{j}\right)\big|_{u_{j}^{-1}}.
			$$
			\item ($S_p$-operator)
			For any $x \in \widetilde{\mathrm{H}}$, we have
			$$
			S_{p}(\varphi)(x):=\varphi\left(x\begin{bmatrix}
				p^{-1} & 0 \\
				0 & p^{-1}
			\end{bmatrix}\right).
			$$
		\end{itemize} 
		Note that the operator $S_{p}$ is invertible and commutes with $T_{p}$. Hence, $S_{\tH, k}^{\mathrm{ur}}\left(\eta\right)$ admits an
		$\mathcal{O}\left[T_{p}, S_{p}^{\pm 1}\right]$-module structure.
		\item[(6)] Let $R_{\overline{r}_{p}}^{\square, 1-k, \eta}$ be the quotient of $R_{\overline{r}_{p}}^{\square}$ parametrizing crystabelian representations with Hodge-Tate weight $(1-k, 0)$ such that $\GQp$ acts on $\mathbb{D}_{\text {pcrys }}(-)$ by $\eta$ (here, we remind the reader that $\mathbb{D}_{\text {pcrys }}(-)$ is the union of $\mathbb{D}_{\text {crys }}(-)$ for $\mathbb{Q}_{p}\left(\zeta_{p^{N}}\right)$ with $N$ sufficiently large.)
		Let $\mathcal{V}_{1-k}$ denote the universal representation on $\mathcal{X}_{\overline{r}_{p}}^{\square, 1-k, \eta}:=\left(\operatorname{Spf} R_{\overline{r}_{p}}^{\square, 1-k, \eta}\right)^{\text {rig }}$, then $\mathbb{D}_{\text {pcrys }}\left(\mathcal{V}_{1-k}\right)$ is locally free of rank two over $\mathcal{X}_{\overline{r}_{p}}^{\square, 1-k, \eta}$, equipped with a linear action of crystalline Frobenius $\phi$.
		Define elements $s_{p} \in \mathcal{O}\left(\mathcal{X}_{\overline{r}_{p}}^{\square, 1-k, \eta}\right)^{\times}$and $t_{p} \in \mathcal{O}\left(\mathcal{X}_{\overline{r}_{p}}^{\square, 1-k, \eta}\right)$ such that
		$$
		\operatorname{det}(\phi)=p^{k-1} s_{p}^{-1} \quad \text { and } \quad \operatorname{tr}(\phi)=s_{p}^{-1} t_{p}.
		$$
		As both $s_{p}$ and $t_{p}$ take bounded values, we have $s_{p} \in R_{\overline{r}_{p}}^{\square, 1-k, \eta}\left[\frac{1}{p}\right]^{\times}$and $t_{p} \in R_{\overline{r}_{p}}^{\square, 1-k, \eta}\left[\frac{1}{p}\right]$. Following \cite[\S4]{ceggps}, we define a natural homomorphism
		$$
		\eta_{k}: \mathcal{O}\left[T_{p}, S_{p}^{\pm 1}\right] \rightarrow R_{\overline{r}_{p}}^{\square, 1-k, \eta}\left[\frac{1}{p}\right] \quad \text { given by } \quad \eta_{k}\left(T_{p}\right)=t_{p}, \text { and } \eta_{k}\left(S_{p}\right)=s_{p}.
		$$
	\end{enumerate}
\end{notation-definition}

\begin{definition}\label{re::1}\hspace{2em}
	\begin{enumerate}
\item		For a character $\ve: \Delta^{2} \rightarrow \mathcal{O}^{\times}$, we define 
		$\ve_1:\Delta\rightarrow \mathcal{O}^{\times}$ and $k^{(\ve)}\in \{2,\dots,p\}$ by 	\begin{equation}\label{eq9}
			\ve=\ve_{1} \times \ve_{1} \omega^{k^{(\ve)}-2}.
		\end{equation}
A pair of  integers $(a,b)$ with $0\leq a\leq p-2$, $0\leq b\leq p-2$ is called \emph{related to $\ve$} if $\ve(\overline{\alpha}, \overline{\alpha})=\omega(\overline{\alpha})^{a+2b}$ for all $\overline{\alpha} \in \Delta$. For every $\ve$-related $(a,b)$, there is a unique 
 $s\in \{0,\dots,p-2\}$ such that $\ve=\omega^{-s+b} \times \omega^{a+s+b}.$ We write this  one-to-one correspondence by $\iota$ which maps  $s\in \{0,\dots,p-2\}$  to an $\ve$-related $(a,b)$.
\end{enumerate}
Let $\Or$ be any hybrid representation  of type $(a,b)$.
\begin{enumerate}
	\item[(2)] If $(a,b)$ is related to $\ve$, then we call $\Or$ \emph{$\ve$-related}. 
		\item[(3)] 
		An $\calO\llbracket \rmK_p\rrbracket$-projective augmented module $\tH$ is called \emph{of  type $\overline{r}$} if it is equipped with a continuous left action of $R_{\overline{r}_{p}}^{\square}$ such that
		\begin{enumerate}
			\item the left $R_{\overline{r}_{p}}^{\square}$-action on $\widetilde{\mathrm{H}}$ commutes with the right $\mathrm{GL}_{2}\left(\mathbb{Q}_{p}\right)$-action
			\item The induced $\rmK_p$-action makes $\widetilde{\mathrm{H}}$ a right $\mathcal{O} \llbracket \rmK_p \rrbracket$-module isomorphic to the direct sum of some number $m(\widetilde{\mathrm{H}}) \in \mathbb{N}$ copies of
			\begin{itemize}
				\item  a projective envelope of $\sigma_{a, b}$ as an $\mathcal{O} \llbracket \mathrm{K}_{p} \rrbracket$-module, if $\Or|_{I_{\QQ_p}}$ is nonsplit, or
				\item a projective envelope of $\sigma_{a, b} \oplus \sigma_{p-3-a, a+b+1}$ as an $\mathcal{O} \llbracket K_{p} \rrbracket$-module, if $\Or|_{I_{\QQ_p}}$ is split.
			\end{itemize}
			\item For every $\Or$-relevant character $\varepsilon$ and integer $k\geq 2$ such taht $k\equiv k^{(\ve)}\pmod {p-1}$, the induced $R_{\overline{r}_{p}}^{\square}$-action on $\mathrm{S}_{\widetilde{\mathrm{H}}, k}^{\mathrm{ur}}\left(\varepsilon_{1}\right)$ factors through the quotient $R_{\overline{r}_{p}}^{\square, 1-k, \varepsilon_{1}}$. Moreover, the Hecke action of $\mathcal{O}\left[T_{p}, S_{p}^{\pm 1}\right]$ on $\mathrm{S}_{\widetilde{\mathrm{H}}, k}^{\mathrm{ur}}\left(\varepsilon_{1}\right)$ agrees with the composition
			$$
			\mathcal{O}\left[T_{p}, S_{p}^{\pm 1}\right] \stackrel{\eta_{k}}{\longrightarrow} R_{\overline{r}_{p}}^{\square, 1-k, \varepsilon_{1}}\left[\frac{1}{p}\right] \rightarrow \operatorname{End}_{E}\left(\mathrm{~S}_{\widetilde{\mathrm{H}}, k}^{\mathrm{ur}}\left(\varepsilon_{1}\right) \otimes_{\mathcal{O}} E\right).
			$$
		\end{enumerate}
		The number $m(\widetilde{\mathrm{H}})$ is called the multiplicity of $\widetilde{\mathrm{H}}$.
	We call
		$\widetilde{\mathrm{H}}$ \emph{primitive} if $m(\widetilde{\mathrm{H}})=1$.
	\end{enumerate}
\end{definition}
\begin{remark}
	Here, for simplicity of notation, instead of using ``type $\Or_p$" as in \cite{xiao1}, we use ``type $\Or$''. It is worthy mentioning that since two representations of types $ \Or_1$ and $\Or_2$ such that $\Or_1|_{\Gal_{\QQ_{p}}}=\Or_2|_{\Gal_{\QQ_{p}}}$ define the same ghost series, our result is compatible to that in   \cite{xiao1}.
\end{remark}

\subsection{Characteristic power series and ghost series}\label{S.22}
From now on, we fix  a character $\ve:=\ve_1\times \ve_2: \Delta^{2} \rightarrow \calO^{\times}$. Write 
$\tv:=\ve_1\times\ve_1$ and 
$ 
\mathcal{O} \llbracket w \rrbracket^{(\ve)}:=\mathcal{O} \llbracket \Delta \times \mathbb{Z}_{p}^{\times} \rrbracket \otimes_{\mathcal{O}\left[\Delta^{2}\right], \ve} \mathcal{O},$
which is isomorphic to the ring of formal power series $\mathcal{O} \llbracket w \rrbracket$ with $w$ corresponding to $[(1, \exp (p))]-1$; and 
$$
\Bo\left(\mathbb{Z}_{p}\right):=\begin{bmatrix}
	\mathbb{Z}_{p}^{\times} & 0 \\
	p\ZZ_p & \mathbb{Z}_{p}^{\times}
\end{bmatrix} \subset	\Iw_{p}:=
\begin{bmatrix}\mathbb{Z}_{p}^{\times} & \mathbb{Z}_{p} \\ p \mathbb{Z}_{p} & \mathbb{Z}_{p}^{\times}\end{bmatrix}\subset
\mathrm{K}_{p}=\mathrm{GL}_{2}\left(\mathbb{Z}_{p}\right).$$
The universal character is defined by
\begin{align*}
	\chi_{\text {univ}}^{(\ve)}:& \;\Bo\left(\mathbb{Z}_{p}\right)\longrightarrow\left(\mathcal{O} \llbracket w \rrbracket^{(\ve)}\right)^{\times} \\
	&	\begin{bmatrix}
		\alpha &0 \\
		\gamma &	\delta
	\end{bmatrix} \longmapsto[(\overline{\alpha}, \delta)] \otimes 1=\ve(\overline{\alpha}, \overline{\delta}) \cdot(1+w)^{\delta / \omega(\overline{\delta})},
\end{align*}
where $\overline{\alpha}$ and $\overline{\delta}$ are the reductions of $\alpha$ and $\delta$ modulo $p$, respectively.



\begin{notation-definition}\label{re:7}\hspace{2em}
	\begin{enumerate}
			\item	For every $\ve$-related pair of integers $(a,b)$, we fix 
		\begin{enumerate}
			\item a hybrid representation $\Or_{(a,b)}$ of type $(a,b)$ such that $\Or_{(a,b)}|_{I_{\QQ_p}}$ is nonsplit; and
			\item a primitive  $\tH_{(a,b)}$ of type $\Or_{(a,b)}$.
		\end{enumerate}
		We put $\HH^{(\ve)}:=\{\tH_{(a,b)}\;|\;(a,b) \textrm{~is an~}  \ve\textrm{-related pair of integers}\}.$
	\end{enumerate}
Let $\Or$ be any $\ve$-related hybrid representation,  and $\tH$ be any 
$\calO\llbracket \rmK_p\rrbracket$-projective augmented module  of type. 
\begin{enumerate}
		\item[(2)]   The \emph{space of abstract $p$-adic forms (w.r.t. $\tH$)} is defined by
	$$
	S_{\widetilde{H}}^{(\ve)}=S_{\widetilde{H},p\adic}^{(\ve)}:=\Hom_{\mathcal{O}\left[\mathrm{Iw}_{p}\right]}\left( \tH, \operatorname{Ind}_{B\left(\mathbb{Z}_{p}\right)}^{\mathrm{Iw}_{p}}\left(\chi_{\text {univ}}^{(\ve)}\right)\right)
	\cong \Hom_{\mathcal{O}\left[\mathrm{Iw}_{p}\right]}\left(	\tH, \mathcal{C}^{0}\left(\mathbb{Z}_{p} ; \mathcal{O} \llbracket w \rrbracket^{(\ve)}\right)\right),
	$$
	which is a module over $\mathcal{O} \llbracket w \rrbracket$, carrying an $\mathcal{O} \llbracket w \rrbracket$-linear $U_{p}$-action as follows: \\
	Fix a decomposition of the double coset $$\operatorname{Iw}_{p}
	\begin{bmatrix}p &0 \\ 0&1\end{bmatrix} \mathrm{Iw}_{p}=\coprod\limits_{j=0}^{p-1} \mathrm{Iw}_{p} u_{j} \quad(\textrm{e.g.~} u_{j}=\begin{bmatrix}
		p & 0 \\ j p & 1\end{bmatrix}).$$ The $U_p$-operator acts on $\varphi \in S_{\widetilde{H}}^{(\ve)}$ by
	$$
	U_{p}(\varphi)(x):=\sum\limits_{j=0}^{p-1} \varphi\left(xu_{j}^{-1}\right)\Big|_{u_{j}} \quad \text { for all } x \in\tH,
	$$
	where for every $h\in  \operatorname{Ind}_{B\left(\mathbb{Z}_{p}\right)}^{\mathrm{Iw}_{p}}\left(\chi_{\text {univ}}^{(\ve)}\right)$ and  every
	$\begin{bmatrix}\alpha &\beta \\ \gamma & \delta\end{bmatrix}\in M_2(\ZZ_p)$ such that $p\,|\, \gamma$, $p \nmid \delta$ and $\alpha \delta-\beta \gamma=p^rd$ with $d \in \mathbb{Z}_{p}^{\times}$,
	$$h\Big|_{\begin{bsmallmatrix}\alpha &\beta \\ \gamma & \delta\end{bsmallmatrix}}(z):=\ve(\overline{d}/\overline\delta, \overline{\delta}) \cdot(1+w)^{\log ((\gamma z+\delta) / \omega(\overline{\delta})) / p} \cdot h\left(\frac{\alpha z+\beta}{\gamma z+\delta}\right).$$

	\item[(3)] We denote by $C_{\tH}^{(\ve)}(w,-)$ the characteristic power series of $U_p$ acting on 		$S_{\tH}^{(\ve)}.$

%

\item[(4)] Similar to the definition of $\mathrm{S}_{\tH,k}^{\mathrm{ur}}(\eta)$ in Notation-Definition~\ref{re::2}(5),
for every $k\geq 2$, we define  $$\mathrm{S}_{\tH,k}^{\mathrm{Iw}}(\tv):=\Hom_{\mathcal{O}\left[\mathrm{Iw}_{p}\right]}\left(\widetilde{\mathrm{H}}, \mathcal{O}[z]^{\leq k-2} \otimes \tv\right),$$
and write  $$d_{\tH,k}^{\mathrm{ur}}(\eta):=\rank_{\mathcal{O}} \left(\mathrm{S}_{\tH,k}^{\mathrm{ur}}(\eta)\right) \quad\textrm{and}\quad 
d_{\tH,k}^{\mathrm{Iw}}(\tv):=\rank_{\mathcal{O}}\left( \mathrm{S}_{\tH,k}^{\mathrm{Iw}}(\tv)\right).$$
Note that $\mathrm{S}_{\tH,k}^{\mathrm{ur}}(\eta)$ and $\mathrm{S}_{\tH,k}^{\mathrm{Iw}}(\tv)$
depend only on $(a,b)$, $m(\tH)$, $k$ and $\rI$ being split or nonsplit. In particular, for any $\ve$-related $(a,b)$, we put 
$$d_{(a,b), k}^{\mathrm{ur}}(\eta):=d_{\tH_{(a,b)},k}^{\mathrm{ur}}(\eta)\quad \textrm{and}\quad d_{(a,b), k}^{\mathrm{Iw}}(\tv):=d_{\tH_{(a,b)},k}^{\mathrm{Iw}}(\tv).$$

	\item[(5)]	Let $\tH_1,\dots, \tH_u$  be 
	 a sequence of 
	$\calO\llbracket \rmK_p\rrbracket$-projective augmented modules
of $\ve$-related type $\Or_i$ for $i=1,\dots,u$. For $\tH:=\bigoplus\limits_{i=1}^u \tH_i$ and any integer $k\geq 2$,  we define
	$$S_{\tH,k}^{\ur}(\ve_1):=\Hom_{\calO[K_p]}\left(\tH, \mathcal{O}[z]^{\leq k-2} \otimes \ve_1 \circ \mathrm{det}\right)\textrm{~and~}
	d_{\tH,k}^{\mathrm{ur}}(\ve_1):=\rank_{\mathcal{O}} \left(S_{\tH,k}^{\mathrm{ur}}(\ve_1)\right).$$
	 Note that from the isomorphism 
$	S_{\tH,k}^{\ur}(\ve_1)
	\cong
	\bigoplus\limits_{i=1}^uS_{\tH_i,k}^{\ur}(\ve_1),$
	we have \begin{equation}\label{re::6}
			d_{\tH,k}^{\ur}(\ve_1)=\sum\limits_{i=1}^u	d_{\tH_i,k}^{\ur}(\ve_1).
	\end{equation}
	
	Similarly, for our fixed character $\ve$, let
	$$S_{\tH,k}^{\mathrm{Iw}}(\tv):=\Hom_{\mathcal{O}\left[\mathrm{Iw}_{p}\right]}\left(\tH, \mathcal{O}[z]^{\leq k-2} \otimes \tv\right) \textrm{~and~} 
	d_{\tH,k}^{\mathrm{Iw}}(\tv):=\rank_{\mathcal{O}} \left(S_{\tH,k}^{\mathrm{Iw}}(\tv)\right),$$
	and obtain
 \begin{equation}\label{re::7}	d_{\tH,k}^{\Iw}(\tv)=\sum\limits_{i=1}^u	d_{\tH_i,k}^{\Iw}(\tv).	
\end{equation}

	\item[(6)] For a power series $f(t)=\sum\limits_{n=0}^\infty a_nt^n\in \CC_p[\![t]\!]$, the lower convex hull of points $\{\left(n, v_{p}\left(a_n\right)\right)\}_{n\geq 0}$ is called the \emph{Newton polygon of $f$}, and denoted by $\mathrm{NP}(f)$.

	\end{enumerate}
\end{notation-definition}

\begin{definition}\label{re:3}
	Let $\Or_1,\dots, \Or_u$ be a sequence of $\ve$-related hybrid representations of $\Gal_\QQ$, and  $\tH_i$ (for $1\leq i\leq u$) be 
	$\calO\llbracket \rmK_p\rrbracket$-projective augmented modules of  type $\Or_i$.
	For  
	$\tH:=\bigoplus\limits_{i=1}^u \tH_i$, we  call  $G^{(\ve)}_{\tH}(w, t):=\sum\limits_{n \geq 0} g^{(\ve)}_{\tH,n}(w) t^n \in \mathbb{Z}_{p}[w] \llbracket t \rrbracket$ \emph{the ghost series corresponding to $(\tH, \ve)$}, where 
	\begin{equation*}
		g^{(\ve)}_{\tH,n}(w):=\prod_{k \equiv k^{(\ve)} \bmod (p-1)}\left(w-w_{k}\right)^{m_{\tH,n}^{(\ve)}(k)} \in \mathbb{Z}_{p}[w],
	\end{equation*}
$w_{k}:=\exp ((k-2) p)-1$, and the exponents $m_{\tH,n}^{(\ve)}(k)$ are given by the following recipe:
	\begin{equation*}
		m_{\tH,n}^{(\ve)}(k):=\begin{cases}
			\min \left\{n-d_{\tH,k}^{\ur}(\ve_1), d_{\tH,k}^{\Iw}(\tv)-d_{\tH, k}^{\ur}(\ve_1)-n\right\}& \textrm{for~} d_{\tH,k}^{\ur}(\ve_1)<n<d_{\tH,k}^{\Iw}(\tv)-d_{\tH,k}^{\ur}(\ve_1),\\
			0&\textrm{else}.
		\end{cases}
	\end{equation*}
	In particular, for every $\ve$-related $(a,b)$, we write $G^{(\ve)}_{(a,b)}(w,t):=G^{(\ve)}_{\tH_{(a,b)}}(w,t)$, which, by Notation-Definition~\ref{re:7}(4), is independent of the choices of $\Or_{(a,b)}$ and $\tH_{(a,b)}$.

\end{definition}

\section{The zigzag criterion}
We write $k_0:=k^{(\ve)}$ and 
$\calK:=\{k\geq 2\;|\;k\equiv k_0\pmod {p-1}\}$.

\begin{lemma}\label{re:1}
	For any $\ve$-related hybrid  representation $\Or$ of type $(a,b)$ and an  $\calO\llbracket \rmK_p\rrbracket$-projective augmented module $\tH$ of type $\Or$ with multiplicity $m$, we have
	$$d_{\tH,k}^{\ur}(\ve_1)=\begin{cases}
		md_{(a,b),k}^{\ur}(\ve_1)& \textrm{if}\  \Or|_{I_{\QQ_p}}\ \textrm{is nonsplit},\\
			m\left(d_{(a,b),k}^{\ur}(\ve_1)+d_{(p-3-a,a+b+1),k}^{\ur}(\ve_1)\right)& \textrm{if}\  \Or|_{I_{\QQ_p}} \ \textrm{is split},
	\end{cases}
	$$	
	and 
	$$d_{\tH,k}^{\Iw}(\tv)=\begin{cases}
	md_{(a,b),k}^{\ur}(\tv)& \textrm{if}\  \Or|_{I_{\QQ_p}}\ \textrm{is nonsplit},\\
	m\left(d_{(a,b),k}^{\Iw}(\tv)+d_{(p-3-a,a+b+1),k}^{\Iw}(\tv)\right)& \textrm{if}\  \Or|_{I_{\QQ_p}} \ \textrm{is split}.
\end{cases}
$$	
\end{lemma}
\begin{proof}
	This follows from the definition of the multiplicity $ m$ of $\tH$.
\end{proof}

\begin{proposition}\label{re:4}
If Theorem~\ref{keythm} holds for modules in $\HH^{(\ve)}$, so does for the general ones. 
\end{proposition}	
\begin{proof}
	Without loss of generality, we may assume that $\ve$ is fixed in \S\ref{S.22}. Let $\Or_i$ and $\tH_i$ for $i=1,\dots,u$ be  stated in Theorem~\ref{keythm}.
	Let $T:=\{i\in\{1,\dots,u\}\;|\;\Or_i|_{I_{\QQ_p}}\ \textrm{is nonsplit}\},$ $T':=\{1,\dots, u\}\backslash T,$
	and $a_i':=p-3-a_i$, $b_i':=a_i+b_i+1$ for every $i\in T'$.
	By \eqref{re::6}, \eqref{re::7} and Lemma~\ref{re:1}, for any $1\leq i\leq u$, we have 
		\begin{equation*}
		\NP\left(G^{(\ve)}_{\tH_i}(w_\star,-)\right)=
		\begin{cases}
				\NP\left(G^{(\ve)}_{  \tH_{(a_i,b_i)}^{\oplus m_i}}(w_\star,-)\right)&
				\textrm{if}\  i\in T,\\
					\NP\left(G^{(\ve)}_{  \tH_{(a_i,b_i)}^{\oplus m_i}\oplus \tH_{(a_i',b_i')}^{\oplus m_i}}(w_\star,-)\right)&
						\textrm{if}\  i\in T',
		\end{cases}
	\end{equation*}
and
	\begin{equation}\label{neweq1}
		\NP\left(G^{(\ve)}_{\tH}(w_\star,-)\right)=
		\NP\left(G^{(\ve)}_{\tH'}(w_\star,-)\right),
	\end{equation}
where $\tH':=\bigoplus\limits_{i\in T} 
\left( \tH_{(a_i,b_i)}^{\oplus m_i}\right)\oplus \bigoplus\limits_{j\in T'} 
\left(\tH_{(a_j,b_j)}^{\oplus m_j}\oplus \tH_{(a_j',b_j')}^{\oplus m_j}\right)$.

Combining
\begin{itemize}
	\item the first equality, 
	\item the hypothesis of this proposition that the condition (2) implies (1) for modules in $\HH^{(\ve)}$, 
	\item the fact that $(a_j,b_j)$ and $(a_j', b_j')$
	satisfy the second relation in (2) of Theorem~\ref{keythm}, 
\end{itemize} we have
		\begin{gather*}
		\NP\left(G^{(\ve)}_{  \tH_{(a_i,b_i)}^{\oplus m_i}}(w_\star,-)\right)
		=
	\NP\left(G^{(\ve)}_{ \tH_{(a_i,b_i)}}(w_\star,-)\right)^{\#m_i},\\
		\NP\left(G^{(\ve)}_{  \tH_{(a_j,b_j)}^{\oplus m_j}\oplus \tH_{(a_j', b_j')}^{\oplus m_j}}(w_\star,-)\right)=
		\NP\left(G^{(\ve)}_{ \tH_{(a_j,b_j)}}(w_\star,-)\right)^{\#m_j}\#
	\NP\left(G^{(\ve)}_{ \tH_{(a_j', b_j')}}(w_\star,-)\right)^{\#m_j},
\end{gather*} 
and hence 
\begin{align*}
&	\underset{i\in T}{\mathlarger{\#}}
	\NP\left(G^{(\ve)}_{  \tH_{(a_i,b_i)}^{\oplus m_i}}(w_\star,-)\right){ \#}
\underset{j\in T'}{\mathlarger{\#}}
	\NP\left(G^{(\ve)}_{  \tH_{(a_j,b_j)}^{\oplus m_j}\oplus \tH_{(a_j', b_j')}^{\oplus m_j}}(w_\star,-)\right)
\\
\notag=&\underset{i\in T}{\mathlarger{\#}}	\NP\left(G^{(\ve)}_{ \tH_{(a_i,b_i)}}(w_\star,-)\right)^{\#m_i}
\#
\underset{j\in T'}{\mathlarger{\#}}
\left(	\NP\left(G^{(\ve)}_{ \tH_{(a_j,b_j)}}(w_\star,-)\right)^{\#m_j}\#
\NP\left(G^{(\ve)}_{ \tH_{(a_j', b_j')}}(w_\star,-)\right)^{\#m_j}\right).
\end{align*}
 Combined with \eqref{neweq1}, this equality implies that the condition (1) of Theorem~\ref{keythm}
  holds for $\tH_1,\dots, \tH_u$ if and only if 
  the right hand side of the above equality is equal to 
  $	\NP\left(G^{(\ve)}_{\tH'}(w_\star,-)\right).$
  By the hypothesis of this proposition that the two conditions of Theorem~\ref{keythm} 
  are equivalent for modules in $\HH^{(\ve)}$, this equality is exactly equivalent to the condition (2) of  Theorem~\ref{keythm}.
This completes the proof.
%
%
%
%
%
\end{proof}
	
	Hence, without of loss generality, from now on we assume that  $\tH_i$'s belong to the set $\HH^{(\ve)}$.

\begin{notation}\label{re:2}
(1) By Definition~\ref{re::1}(1), we parameterize the modules in $\HH^{(\ve)}$ by $\{0,\dots, p-2\}$ 
with
$\tH(s):=\tH_{\iota(s)}$.
If we write $\iota(s)=(a,b)$, then  \begin{equation}\label{eq5}
	a+2s\equiv k_0-2 \pmod {p-1}.
\end{equation}

(2) For any $\us:=(s_1,\dots,s_u)\in \{0,\dots,p-2\}^u$ and any $k\in \calK$, we put 
		\begin{gather*}
		d_{k}^{\ur}\left(\us\right):=d_{\tH,k}^{\ur}\left(\ve_1\right),\quad
		d_{k}^{\Iw}\left(\us\right):=d_{\tH,k}^{\Iw}\left(\tv\right), \quad m_n^{(\us)}(k):=	m_{\tH,n}^{(\ve)}(k), \\
		g^{(\us)}_{n}(w):=g^{(\ve)}_{\tH,n}(w)
		 \quad\textrm{and}\quad 
		 G^{(\us)}(w, t):= G_{\tH}^{(\ve)}(w, t),
		\end{gather*}
		where $\tH:=\bigoplus\limits_{i=1}^u\tH(s_i)$.
		Clearly, for every $n\geq 0$, we have	
		\begin{itemize}
			\item 	
			$	m_n^{(\us)}(k)=\begin{cases}
					\min \left\{n-d_{k}^{\ur}\left(\us\right), 	d_{k}^{\Iw}\left(\us\right)-d_{k}^{\ur}\left(\us\right)-n\right\}& \textrm{for~} d_{k}^{\ur}\left(\us\right)<n<	d_{k}^{\Iw}\left(\us\right)-d_{k}^{\ur}\left(\us\right),\\
					0&\textrm{else},
				\end{cases}$
			
			\item $
				g^{(\us)}_{n}(w)=\prod\limits_{k \equiv k_0 \bmod (p-1)}\left(w-w_{k}\right)^{m_{n}^{(\us)}(k)} \in \mathbb{Z}_{p}[w],
				$
		\item
			$G^{(\us)}(w, t)=\sum\limits_{n \geq 0} g^{(\us)}_{n}(w) t^n \in \mathbb{Z}_{p}[w] \llbracket t \rrbracket$.
		\end{itemize}
		
\end{notation}
Now we interpolate the relation (2) of Theorem~\ref{keythm} to a statement with respect to $\tH(s)$'s. 

\begin{lemma}\label{re:14}
	For $s, s'\in \{0,\dots,p-2\}$, if we write $\iota(s)=(a,b)$ and $\iota(s')=(a',b')$, then 
	$$s+s'\equiv k_0-1\pmod {p-1}\quad \textrm{is equivalent to}\quad a'=\{p-3-a\}\textrm{~and~} b'=\{a+b+1\}.$$
\end{lemma}
\begin{proof}
	Assume that $s+s'\equiv k_0-1\pmod {p-1}$. 
By \eqref{eq5}, we have
	$a+a'\equiv p-3\pmod {p-1}$. 
	We write $\ve=\omega^{c}\times \omega^d$, which is equal to $\omega^{-s+b}\times \omega^{s+a+b}=\omega^{-s+b}\times \omega^{-s+b+k_0-2}$, 
	 and hence obtain  that $b-s\equiv b'-s'\equiv c\pmod {p-1}$. Combined with 
	$2c=c+d+(c-d)\equiv a+2b+(2-k_0)\pmod{p-1}$, this chain of congruence relations implies 
	$$b+b'\equiv 2c+(s+s')\equiv 2c+(k_0-1)\equiv a+2b+1\pmod {p-1}.$$
	This completes the proof.
	
	Now we prove the other direction. Assume $a'=\{p-3-a\}$ and $b'=\{a+b+1\}$. As in previous discussion, we have $$s+s'\equiv (b+b')-2c\equiv (a+2b+1)-2c\equiv (c+d+1)-2c=d-c+1\equiv k_0-2+1=k_0-1\pmod{p-1}.$$
	This completes the proof. 
\end{proof}

%

\begin{notation}\label{notation:2}\hspace{2em}
	\begin{enumerate}
		\item 	For any $k\in \calK$, we write 
		$k_\bullet:=\frac{k-k_0}{p-1}$, i.e., 
		 $k:=(p-1)k_\bullet+k_0$.
		\item 
		For any parameter $s\in \{0,\dots, p-2\}$, we put $a_{s}:=\{k_0-2-2s\}$
and
		$$
		\delta_{s}:=\left\lfloor\frac{s+\left\{a_{s}+s\right\}}{p-1}\right\rfloor= \begin{cases}0 & \text { if } s+\left\{a_{s}+s\right\}<p-1,\\ 
			1 & \text { if } s+\left\{a_{s}+s\right\} \geq p-1.
		\end{cases}
		$$	
	It is easy to verify that $\iota(s)=(a_s,b)$ for some $b\in \{0,\dots,p-2\}$.
		\item When $a_{s}+s<p-1,$ we put $t_{1}^{(s)}:=s+\delta_{s}$ and $t_{2}^{(s)}:=a_{s}+s+\delta_{s}+2$.
		\item When $a_{s}+s \geq p-1, $ we put $t_{1}^{(s)}:=\left\{a_{s}+s\right\}+\delta_{s}+1$ and $t_{2}^{(s)}:=s+\delta_{s}+1$.	
		\item For any $n\in \ZZ$,  we write  $$\beta_{n}^{(s)}:=\left\{\begin{array}{ll}
			t_{1}^{(s)} & \text { if } n \text { is even,} \\
			t_{2}^{(s)}-\frac{p+1}{2} & \text { if } n \text { is odd.}
		\end{array}\right. $$
	\end{enumerate}
\end{notation}

\begin{lemma}\label{re:10}
	For every $k\in \calK$ and  every $s\in  \{0,\dots, p-2\}$, we have
	\begin{equation}
		d_{k}^{\Iw}\left(s\right)=2 k_\bullet+2-2 \delta_{s},
	\end{equation}
	and 
	\begin{equation*}
		d_{k}^{\ur}\left(s\right) =\left\lfloor\frac{k_\bullet-t_{1}^{(s)}}{p+1}\right\rfloor+\left\lfloor\frac{k_\bullet-t_{2}^{(s)}}{p+1}\right\rfloor+2.
	\end{equation*}
\end{lemma}

\begin{proof}
	By Notation~\ref{notation:2}, we have  $\tH(s)=\tH_{(a_s, b)}$ for some $b\in \{0,\dots,p-2\}$. Let $\ve'=\ve\cdot(\omega^{-b}\times \omega^{-b})=\omega^{-s} \times \omega^{a_s+s}$. Twisting $\ve $ by $\omega^{-b}\times \omega^{-b}$, 
	we have
\begin{align}\label{neweq2}
		S_{\tH(s),k}^{\Iw}(\tv)=&\Hom_{\calO\left[\Iw_{p}\right]}\left(\tH(s), \calO[z]^{\leq k-2} \otimes \tv\right) \\\notag
		\cong &\Hom_{\calO\left[\Iw_{p}\right]}\left(\tH', \calO[z]^{\leq k-2} \otimes (\omega^{-s} \times \omega^{-s})\right)\\=&S_{\tH',k}^{\Iw}(\omega^{-s} \times \omega^{-s}),\notag
\end{align}
and 
 \begin{align}\label{neweq3}	S_{\tH(s),k}^{\ur}(\ve_1)=&\Hom_{\calO\left[K_p\right]}\left(\tH(s), \calO[z]^{\leq k-2} \otimes \ve_1\circ\det\right) \\
\notag	\cong &\Hom_{\calO\left[K_p\right]}\left(\tH', \calO[z]^{\leq k-2} \otimes \omega^{-s}\circ\det\right)\\=&S_{\tH',k}^{\Iw}(\omega^{-s}).\notag
\end{align}
	Here $\widetilde{\ve'}=\omega^{-s}\times\omega^{-s}$, and $\tH'$ is a primitive $\calO\llbracket \rmK_p\rrbracket$-projective augmented module of type $\Or'$, where $\Or'$ is an $\ve'$-related hybrid representation of type $(a_s, 0)$.
	By \cite[Corollary~4.4]{xiao}, we have
	$$\dim_{\calO} \left(S_{\tH',k}^{\Iw}(\omega^{-s}\times\omega^{-s})\right)=2 k_\bullet+2-2 \delta_{s}.$$
	Combined with \eqref{neweq2}, this implies that
		$$d_{k}^{\Iw}\left(s\right)=\dim_{\calO} \left(S_{\tH(s),k}^{\Iw}(\tv)\right)=2 k_\bullet+2-2 \delta_{s}.$$
		By \cite[Proposition~4.7]{xiao}, we have
			$$\dim_{\calO} \left(S_{\tH',k}^{\ur}(\omega^{-s})\right)=\left\lfloor\frac{k_\bullet-t_{1}^{(s)}}{p+1}\right\rfloor+\left\lfloor\frac{k_\bullet-t_{2}^{(s)}}{p+1}\right\rfloor+2.$$
			Combined with \eqref{neweq3}, this implies that
		\[d_{k}^{\ur}\left(s\right)=\dim_{\calO} \left(S_{\tH(s),k}^{\ur}(\ve_1)\right)=\left\lfloor\frac{k_\bullet-t_{1}^{(s)}}{p+1}\right\rfloor+\left\lfloor\frac{k_\bullet-t_{2}^{(s)}}{p+1}\right\rfloor+2.\qedhere\]
\end{proof}

\begin{lemma}\label{lem:1}
We have the following table:
	\begin{center}
		\begin{tabular}{|c|c|c|c|c|c|}
			\hline
			$s $& $\delta_{s}$ & $t_{1}^{(s)}$ & $t_{2}^{(s)}$  \\
			\hline
			$	\{0, \dots, \lfloor \frac{k_0-2}{2}\rfloor  \}$ & $0 $  &   $s$  & $k_0-s$     \\
			\hline
			$	\{\lfloor \frac{k_0-2}{2}\rfloor+1, \dots, k_0-2\}$     &  $0 $  &  $k_0-s-1$ & $s+1$    \\
			\hline
			$	\{k_0-1, \dots, \lfloor \frac{k_0-2+p-1}{2}\rfloor \}$    & $1$   & $s+1$  & $p+k_0-s$             \\
			\hline
			$	\{\lfloor \frac{k_0-2+p-1}{2}\rfloor+1, \dots, p-2\}$ & $1$&  $k_0-s+p-1$    &$s+2$ \\
			\hline
			
		\end{tabular}.
	\end{center}

\end{lemma}
\begin{proof}
	(1)	For $s\in 	\{0, 1, \dots, \lfloor \frac{k_0-2}{2}\rfloor \}$, we have $a_{s}=k_0-2-2s$, and hence 
	$$a_{s}+s=k_0-2-s<p-1, \ \delta_{s}=\left\lfloor\frac{k_0-2}{p-1}\right\rfloor =0,\  t_{1}^{(s)} = s,\  t_{2}^{(s)}=k_0-s. $$
	
	(2)	For $s\in 	\{\lfloor \frac{k_0-2}{2}\rfloor+1, \dots, k_0-2\}$, we have $a_{s}=k_0+p-1-2-2s$, and hence 
	$$a_{s}+s=k_0-2-s+p-1\geq  p-1, \ \delta_{s}=\left\lfloor\frac{k_0-2}{p-1}\right\rfloor =0,\  t_{1}^{(s)} = k_0-s-1,\  t_{2}^{(s)}=s. $$
	
	(3)	For $s\in 	\{k_0-1, \dots, \lfloor \frac{k_0-2+p-1}{2}\rfloor \}$, we have $a_{s}=k_0+p-1-2-2s$, and hence 
	$$a_{s}+s=k_0-2-s+p-1< p-1, \ \delta_{s}=\left\lfloor\frac{k_0-3+p}{p-1}\right\rfloor =1,\  t_{1}^{(s)} = s+1,\  t_{2}^{(s)}=k_0-s+p. $$
	
	(4)	For $s\in 	\{ \lfloor \frac{k_0-2+p-1}{2}\rfloor +1, \dots, p-2\}$, we have $a_{s}=k_0+2(p-1)-2-2s$, and hence 
	\[a_{s}+s=k_0-2-s+2(p-1)\geq  p-1, \ \delta_{s}=\left\lfloor\frac{k_0-3+p}{p-1}\right\rfloor =1,\  t_{1}^{(s)} = k_0-s+p-1,\  t_{2}^{(s)}=s+2. \qedhere\]
\end{proof}

\begin{notation}\label{eq:1}
	For any $s\in \{1,\dots, p-2\}$,  we normalize  	$$d_{k}^{\ur,\dagger}(s):=	d_{k}^{\ur}\left(s\right) +\delta_{s},$$
	and write
		$$d_{k}^{\Iw,\dagger}:=2k_\bullet+2.$$
		Note that for every $s\in \{0,\dots, p-2\}$, we have
	\begin{equation}
		d_{k}^{\Iw,\dagger}=	d_{k}^{\Iw}\left(s\right) +2\delta_{s}.
	\end{equation}

\end{notation}

\begin{notation}
	For a sequence $\underline s=(s_1,s_2,\dots,s_u)\in \{0,1,\dots,p-2\}^u$, we put $$d_k^{\ur,\dagger}(\us):=\sum\limits_{i=1}^u d_k^{\ur,\dagger}(s_i)\quad \textrm{and}\quad d_k^{\Iw,\dagger}(\us):=u\cdot  d_k^{\Iw,\dagger}.$$
	

\end{notation}	
	
	Consider Notation~\ref{re:2}. 
	By replacing  $d_k^{\ur}(\us)$ and $d_k^{\Iw}(\us)$
	with $d_k^{\ur,\dagger}(\us)$ and $d_k^{\ur,\dagger}(\us)$ in Definition~\ref{re:3},	we define  $m_n^{(\us),\dagger}(k), g_{n}^{(\us),\dagger}(w), G^{(\us),\dagger}(w,-)$. 
	We simply replace $\us$ by $s$ if $\us$ has length one.

	\begin{lemma}\label{re:5}
For any $w_\star\in \bfm_{\CC_p}$,	the Newton polygon $\NP\left(G^{(\us),\dagger}(w_\star,-)\right)$ is equal to the polygon that patches a horizontal segment of length $\sum\limits_{i=1}^u \delta_{s_i}$ with  $\NP\left(G^{(\us)}(w_\star,-)\right)$.
	\end{lemma}
\begin{proof}
		By \eqref{re::6}  and  \eqref{re::7}, we have $$d_k^{\ur,\dagger}(\us)=d_k^{\ur}(\us)+\sum\limits_{i=1}^u \delta_{s_i}\quad\textrm{and}\quad d_k^{\Iw,\dagger}(\us)=d_k^{\Iw}(\us)+2\sum\limits_{i=1}^u \delta_{s_i},$$ and hence
	$$g_{n}^{(\us),\dagger}(w)=\begin{cases}
		g_{n-\sum\limits_{i=1}^u (\delta_{s_i})}^{(\us)}(w) &\textrm{for~} n\geq \sum\limits_{i=1}^u \delta_{s_i},
		\\
		1 &\textrm{for~} 0\leq n< \sum\limits_{i=1}^u \delta_{s_i}.
	\end{cases}$$
This completes the proof.
\end{proof}

\begin{remark}
		By Proposition~\ref{re:4},  Notation~\ref{re:2} and Lemma~\ref{re:5}, to prove Theorem~\ref{keythm}, it is enough to study the Newton polygon of the dagger ghost series parameterized by $\us$, i.e., $\NP\left(G^{(\us),\dagger}(w,-)\right)$. 
\end{remark}
	
%

\begin{lemma}\label{lem:2}\hspace{2em}
	\begin{enumerate}
		\item 	For any $s\in \{0,1,2,\dots,\lfloor \frac{k_0-2}{2}\rfloor\}$, if we put $s':=k_0-1-s$, then 
		$$ G^{(s),\dagger}(w,t)=G^{(s'),\dagger}(w,t)$$
		as power series in $\ZZ_p[w][\![t]\!]$.
		\item For any $s\in \{k_0,\dots,\lfloor \frac{k_0-2+p-1}{2}\rfloor\}$,  if we put $s'':=k_0-1-s+p-1$, then
		$$ G^{(s),\dagger}(w,t)=G^{(s''),\dagger}(w,t)$$
			as power series in $\ZZ_p[w][\![t]\!]$.
	\end{enumerate}	
\end{lemma}

\begin{proof}
	(1) For $s\in \{1,2,\dots,\lfloor \frac{k_0-2}{2}\rfloor\}$
	 from the first two rows in the table of Lemma~\ref{lem:1}, we have $t_{1}^{(s)}=t_{1}^{(s')}$ and $t_{2}^{(s)}=t_{2}^{(s')}$. Combined with 	$\delta_{s}=\delta_{s'}=0$ in this case, these two equalities imply
	 $m_n^{(s),\dagger}(k)=m_n^{(s'),\dagger}(k)$, and further $G^{(s),\dagger}(w,t)=G^{(s'),\dagger}(w,t)$
	 by definition.

	For $s=0$,  from the first and third rows in the table of Lemma~\ref{lem:1}, we have $t_{1}^{(s)}=0$,  $t_{2}^{(s)}=t_{1}^{(s')}=k_0$ and $t_{2}^{(s')}=p+1$, which together imply that for every $k\in \calK$, we have 
	$d_k^{\ur}(s)=d_k^{\ur}(s')-1$. Combined with 	$\delta_{s}=0$ and $\delta_{s'}=1$ in this case, we have
	$m_n^{(s),\dagger}(k)=m_n^{(s'),\dagger}(k)$. This clearly implies $G^{(s),\dagger}(w,t)=G^{(s'),\dagger}(w,t)$
	by definition.

	(2) 
	Similar to (1), by the last two rows in the table of Lemma~\ref{lem:1}, we have $t_{1}^{(s)}=t_{1}^{(s'')}$ and $t_{2}^{(s)}=t_{2}^{(s'')}$. Combined with 	$\delta_{s}=\delta_{s''}=1$ in this case, these two equalities imply
	$m_n^{(s),\dagger}(k)=m_n^{(s''),\dagger}(k)$, and further $G^{(s),\dagger}(w,t)=G^{(s''),\dagger}(w,t)$
	by definition.
\end{proof}

%
%
%

\begin{notation}\label{notation:3}
 We put $$S:=\left\{\left\lceil \frac{k_0+1}{2}\right\rceil,\left\lceil \frac{k_0+1}{2}\right\rceil+1,\dots,\left\lfloor\frac{k_0-4+p}{2}\right\rfloor\right\}.$$
\end{notation}

\begin{lemma}\label{re:15}
For $s\in \{\left\lceil\frac{k_0-1}{2}\right\rceil,\dots, \left\lfloor\frac{p-1+k_0-1}{2}\right\rfloor\}$, 
$\iota(s)$ is generic if and only if $s\in S$.
\end{lemma}	
\begin{proof}
	 For $s\in \{\left\lceil\frac{k_0-1}{2}\right\rceil,\dots, \left\lfloor\frac{p-1+k_0-1}{2}\right\rfloor\}$, we have $a_{s}=p-1+k_0-2-2s$. By \eqref{neq::23},
	we have $1\leq a_s\leq p-4$ is equivalent to $1\leq p-1+k_0-2-2s\leq p-4$, and further to 
	$s\in S$. 	
\end{proof}

%
%

%
%


\begin{definition}\label{def:1}
Given a map $\phi: S\to \ZZ$, we call a pair of indices $s< s'$ in $S$ \emph{$\phi$-regular} if it satisfies that
	\begin{enumerate}
		\item[(a)] $|\phi(s)-\phi(s')|\leq 1$;
		\item[(b)] if $\phi(s)\neq \phi(s')$, then $\phi(s')$ is odd.
	\end{enumerate}

We call $\phi$ \emph{odd dominant (w.r.t. $\ve$)} if every pair of indices $s< s'$ in  $S$ is $\phi$-regular.
\end{definition}
Intuitively speaking, we use ``odd dominant'' to emphasize the property of $\phi$ that $\phi(s)$ is odd for larger $s$ if the smaller one is.

\begin{lemma}\label{lem:5}
	A map $\phi: S\to \ZZ$ is odd dominant if and only if $s_1<s_2$ is $\phi$-regular for all consecutive $s_1$ and $s_2$ in $S.$
\end{lemma}

\begin{proof}
	It is trivial that the first statement implies the second. Now we prove the opposite. Clearly, it is enough to show that for any $s_1< s_2< s_3$, if $s_1< s_2$ and $s_2< s_3$ are both $\phi$-regular, so is $s_1< s_3$. 
	Its proof falls into two cases. (1) If $\phi(s_2)$ is odd, then since $s_2<s_3$ is $\phi$-regular,  $\phi(s_3)$ must also be odd, and $\phi(s_3)=\phi(s_2)$. Combined with that $s_1< s_2$ is $\phi$-regular, this implies that $s_1< s_3$ is $\phi$-regular.
	(2) If $\phi(s_2)$ is even, then since $s_1<s_2$ is $\phi$-regular, $\phi(s_1)$ must also be even, and $\phi(s_1)=\phi(s_2)$. Combined with that $s_2< s_3$ is $\phi$-regular, this implies that $s_1< s_3$ is $\phi$-regular.
\end{proof}

\begin{lemma}\label{lem:3}
	For any $k\in \calK$, the dimension map $d_k^{\ur,\dagger}: S\to \ZZ$ is odd dominant.
\end{lemma}
\begin{proof} 
	Let $ m\in \{0,\dots, p\}$ be the representative element of the congruence class of $k_\bullet$ modulo $p+1$. 
	It is enough to check the conditions (a) and (b) in Definition~\ref{def:1} for $d_k^{\ur,\dagger}$.
	By Lemma~\ref{lem:5}, we split our proof into three cases. We note first that (a) is trivial in all three cases after listing the four sub-cases for each. Now we focus on proving (b).
	
	\textbf{Case 1.} When $\left\lceil \frac{k_0+1}{2}\right\rceil\leq s<k_0-2$.
	By Lemma~\ref{lem:1}, we have $\delta_{s}=\delta_{s+1}=0$ and  $$t_{1}^{(s+1)}< t_{1}^{(s)}< t_{2}^{(s)}<t_{2}^{(s+1)}.$$
	This chain of inequalities implies that
	\begin{enumerate}
		\item if $ m<t_{1}^{(s+1)}$ or $\geq t_{2}^{(s+1)}$, then $d_k^{\ur,\dagger}(s+1)= d_k^{\ur,\dagger}(s)$;
		\item  if $t_{1}^{(s+1)}\leq  m< t_{1}^{(s)}$, then $d_k^{\ur,\dagger}(s+1)= d_k^{\ur,\dagger}(s)+1$;
		\item  if $t_{1}^{(s)}\leq m< t_{2}^{(s)}$, then $d_k^{\ur,\dagger}(s+1)= d_k^{\ur,\dagger}(s)$;
		\item   if $t_{2}^{(s)}\leq m< t_{2}^{(s+1)}$, then $d_k^{\ur,\dagger}(s+1)= d_k^{\ur,\dagger}(s)-1$.
	\end{enumerate}
	Note that in the subcases (2) and (4), when  $d_k^{\ur,\dagger}(s)\neq d_k^{\ur,\dagger}(s+1)$, the dimension $d_k^{\ur,\dagger}(s+1)$ is odd. This proves (b)  for this case.
	
	\textbf{Case 2.}
	By Lemma~\ref{lem:1}, we have $\delta_{k_0-2}=0$,  $\delta_{k_0-1}=1$ and $$t_{1}^{(k_0-2)}< t_{2}^{(k_0-2)}< t_{1}^{(k_0-1)}< t_{2}^{(k_0-1)}.$$
	This chain of inequalities implies that
	\begin{enumerate}
		\item if $ m<t_{1}^{(k_0-2)}$ or $\geq t_{2}^{(k_0-1)}$, then $d_k^{\ur,\dagger}(k_0-1)= d_k^{\ur,\dagger}(k_0-2)+1$;
		\item  if $t_{1}^{(k_0-2)}\leq m< t_{2}^{(k_0-2)}$, then $d_k^{\ur,\dagger}(k_0-1)= d_k^{\ur,\dagger}(k_0-2)$;
		\item  if $t_{2}^{(k_0-2)}\leq m< t_{1}^{(k_0-1)}$, then $d_k^{\ur,\dagger}(k_0-1)= d_k^{\ur,\dagger}(k_0-2)-1$;
		\item   if $t_{1}^{(k_0-1)}\leq m< t_{2}^{(k_0-1)}$, then $d_k^{\ur,\dagger}(k_0-1)= d_k^{\ur,\dagger}(k_0-2)$.
	\end{enumerate}
	Note that in the subcases (1) and (3), when  $d_k^{\ur,\dagger}(k_0-2)\neq d_k^{\ur,\dagger}(k_0-1)$, the dimension $d_k^{\ur,\dagger}(k_0-1)$ is odd. This proves (b) for this case.

	\textbf{Case 3.} When $k_0-1\leq s\leq \left\lfloor\frac{k_0-4+p}{2}\right\rfloor$.
	By Lemma~\ref{lem:1}, we have $\delta_{s}=\delta_{s+1}=1$ and $$t_{1}^{(s)}< t_{1}^{(s+1)}< t_{2}^{(s+1)}<t_{2}^{(s)}.$$
	This chain of inequalities implies that
	\begin{enumerate}
		\item if $ m<t_{1}^{(s)}$ or $\geq t_{2}^{(s)}$, then $d_k^{\ur,\dagger}(s+1)= d_k^{\ur,\dagger}(s)$;
		\item  if $t_{1}^{(s)}\leq  m< t_{1}^{(s+1)}$, then $d_k^{\ur,\dagger}(s+1)= d_k^{\ur,\dagger}(s)-1$;
		\item  if $t_{1}^{(s+1)}\leq m< t_{2}^{(s+1)}$, then $d_k^{\ur,\dagger}(s+1)= d_k^{\ur,\dagger}(s)$;
		\item   if $t_{2}^{(s+1)}\leq m< t_{2}^{(s)}$, then $d_k^{\ur,\dagger}(s+1)= d_k^{\ur,\dagger}(s)+1$.
	\end{enumerate}
Note that in the subcases (2) and (4), when  $d_k^{\ur,\dagger}(s)\neq d_k^{\ur,\dagger}(s+1)$, the dimension $d_k^{\ur,\dagger}(s+1)$ is odd. This proves (b) for this case.
\end{proof}

%
%
%
%

The following partition of an integer has the similar property as in Proposition~\ref{mlemma:Ghost}.
\begin{notation}\label{re:11}
Given any $n\geq 0$ and $u\in \ZZ_{>0}$, we decompose $n$ into $u$ parts $\{n_{u,i}\}_{i=1}^u$ such that $n_{u,i}$ is odd for larger $i$'s and even for smaller $i$'s. More precisely, 
	\begin{itemize}
		\item 	if $\lfloor\frac{n}{u}\rfloor$ is even, we write \[n_{u,i}:=\begin{cases}
			\lfloor\frac{n}{u}\rfloor& \textrm{for~} 1\leq i\leq u(\lfloor\frac{n}{u}\rfloor+1)-n,\\
			\lfloor\frac{n}{u}\rfloor+1&\textrm{for~} u(\lfloor\frac{n}{u}\rfloor+1)-n+1\leq i\leq u;
		\end{cases}\]
		\item 	if $\lfloor\frac{n}{u}\rfloor$ is odd, we write 
		\[n_{u,i}:=\begin{cases}
			\lfloor\frac{n}{u}\rfloor+1& \textrm{for~} 1\leq i\leq n-u\lfloor\frac{n}{u}\rfloor,\\
			\lfloor\frac{n}{u}\rfloor&\textrm{for~} n-u\lfloor\frac{n}{u}\rfloor+1\leq i\leq u.
		\end{cases}\]
	\end{itemize}	 
\end{notation}

\begin{lemma}\label{lem:4}
	Let $\phi$ be any odd dominant map, and $\us=(s_1,\dots, s_u)\in S^u$ be any non-decreasing sequence. 
Then for an integer $n\geq 0$, if  $  n_{u,i_0}<\phi(s_{i_0})$ for some $1\leq i_0\leq u$, then $  n_{u,i}\leq \phi(s_{i})$ for every $1\leq i\leq u$.
\end{lemma}
\begin{proof}
	If  $\phi(s_{i_0})-n_{u,i_0}\geq 2$, by Definition~\ref{def:1}(a) and
	\begin{equation}\label{eq:3}
		|n_{u,i}-n_{u, j}|\leq 1 \textrm{~for every~} i,j\in\{1,\dots,u\},
	\end{equation} we immediately prove this lemma. 
	Now we assume that $\phi(s_{i_0})-n_{u,i_0}= 1$, and discuss it in two cases.
	
	\textbf{Case 1.} If $n_{u,i_0}$ is odd, then $\phi(s_{i_0})=n_{u,i_0}+1$ is even. By Definition~\ref{def:1}(b), it follows that $\phi(s_{i})=\phi(s_{i_0})$ for all $i\leq i_0$. 
	Combined with \eqref{eq:3}, this implies that for every $i\leq i_0$,
	we have $$n_{u,i}\leq n_{u,i_0}+1=\phi(s_{i_0})=\phi(s_{i});$$
	and for any $i\geq i_0$, we have 
	$$\phi(s_{i})\geq \phi(s_{i_0})-1=n_{u,i_0}=n_{u,i}.$$
	
	\textbf{Case 2.} If $n_{u,i_0}$ is even, then $\phi(s_{i_0})=n_{u,i_0}+1$ is odd. By Definition~\ref{def:1}(b),  it follows that  $\phi(s_{i})=\phi(s_{i_0})$ for all $i\geq i_0$. 
	Combined with \eqref{eq:3}, this implies that for every $i\geq i_0$
	we have $$n_{u,i}\leq n_{u,i_0}+1=\phi(s_{i_0})=\phi(s_{i});$$
	and for any $i\leq  i_0$, we have 
	\[\phi(s_{i})\geq \phi(s_{i_0})-1=n_{u,i_0}=n_{u,i}.\qedhere\]
\end{proof}


\begin{proposition}\label{mlemma:Ghost}
For any integer $n\geq 0$ and any non-decreasing sequence $\us=(s_1,\dots, s_u)\in S^u$, we have $$g_{n}^{(\us),\dagger}(w)=\prod_{i=1}^{u}g_{n_{u,i}}^{(s_i), \dagger}(w).$$
\end{proposition}

\begin{proof}
	It is enough to prove that for every $k\in \calK$, we have $$m_n^{(\us),\dagger}(k)=\sum\limits_{i=1}^u m_{n_{u,i}}^{(s_i),\dagger}(k).$$
	By Notation~\ref{eq:1}, this lemma can be further reduced to proving that every $k\in \calK$ at least fits in one of the following cases:
	\begin{enumerate}
		\item$n_{u,i}\leq d_k^{\ur,\dagger}(s_{i})$ for every $1\leq i\leq u$, in which case, $$m_n^{(\us),\dagger}(k)=\sum\limits_{i=1}^u m_{n_{u,i}}^{(s_i),\dagger}(k)=0.$$
		\item  $d_k^{\ur,\dagger}(s_{i})\leq n_{u,i}\leq \frac{1}{2}d_k^{\Iw,\dagger}=k_\bullet+1$ for every $1\leq i\leq u$, in which case, 
		$$m_n^{(\us),\dagger}(k)=n-d_k^{\ur,\dagger}(\us)=\sum\limits_{i=1}^u (n_{u,i}-d_k^{\ur,\dagger}(s_i)) =\sum\limits_{i=1}^u m_{n_{u,i}}^{(s_i),\dagger}(k).$$
		\item  $k_\bullet+1\leq n_{u,i}\leq 2(k_\bullet+1)-d_k^{\ur,\dagger}(s_{i})$ for every $1\leq i\leq u$, in which case, 
		$$m_n^{(\us),\dagger}(k)=2u(k_\bullet+1)-d_k^{\ur,\dagger}(\us)-n=\sum\limits_{i=1}^u (2(k_\bullet+1)-d_k^{\ur,\dagger}(s_{i})-n_{u,i}) =\sum\limits_{i=1}^u m_{n_{u,i}}^{(s_i),\dagger}(k).$$
		\item $ n_{u,i}\geq 2(k_\bullet+1)-d_k^{\ur,\dagger}(s_{i})$ for every $1\leq i\leq u$, in which case,
		$$m_n^{(\us),\dagger}(k)=\sum\limits_{i=1}^u m_{n_{u,i}}^{(s_i),\dagger}(k)=0.$$
	\end{enumerate}

	We first prove that $d_k^{\ur,\dagger}(s_{i})$, $2k_\bullet+2-d_k^{\ur,\dagger}(s_{i})$ and the constant function $k_\bullet+1$ are all odd dominant. It is easy to see that if the first map is odd dominant, so is the second one, and that every constant map from $S$ to $\ZZ$ is odd dominant. Hence, we are left to prove that $d_k^{\ur,\dagger}(s_{i})$ is odd dominant, which follows from Lemma~\ref{lem:3} directly. Therefore, if there is $1\leq i_1\leq u$ such that $n_{u,i_1}<d_k^{\ur,\dagger}(s_{i_1})$, then by Lemma~\ref{lem:4}, we have $n_{u,i}\leq d_k^{\ur,\dagger}(s_{i})$
	for every $1\leq i\leq u$. This implies that $k$ satisfies the first case. Now we assume that $n_{u,i}\geq d_k^{\ur,\dagger}(s_{i})$
	for every $1\leq i\leq u$. Similarly, by Lemma~\ref{lem:4}, if there is $1\leq i_2\leq u$ such that $d_k^{\ur,\dagger}(s_{i_2})\leq n_{u,i_2}< k_\bullet+1$, then $d_k^{\ur,\dagger}(s_{i})\leq n_{u,i}\leq k_\bullet+1$
	for every $1\leq i\leq u$. This implies that $k$ satisfies the second case. With a similar argument for the last two cases, we complete the proof of this lemma.
	%
	%
	%
	%
	%
	%
	%
\end{proof}

\begin{lemma}\label{lemma:parity}
	For any integer $n\geq 0$, any $k\in \calK$ and any pair of indices $s<s'$ in $S$, 
	\begin{enumerate}
		\item if $n$ is odd, then $m_{n+1}^{(s'),\dagger}(k)-m_{n}^{(s'),\dagger}(k)\geq m_{n+1}^{(s),\dagger}(k)-m_{n}^{(s),\dagger}(k);$
		\item if $n$ is even, then $m_{n+1}^{(s'),\dagger}(k)-m_{n}^{(s'),\dagger}(k)\leq  m_{n+1}^{(s),\dagger}(k)-m_{n}^{(s),\dagger}(k)$.
	\end{enumerate}
	
\end{lemma}

\begin{proof}
	By Notations~\ref{re:2}(2) and \ref{eq:1}, we have
\begin{equation}\label{eq::8}
		m_{n+1}^{(s),\dagger}(k)-m_{n}^{(s),\dagger}(k)=
	\begin{cases}
		0& \textrm{if~} n\leq d_k^{\ur,\dagger}(s)-1,\\
		1  & \textrm{if~} d_k^{\ur,\dagger}(s)\leq n\leq k_\bullet,\\
		-1& \textrm{if~}k_\bullet+1\leq n\leq  2(k_\bullet+1)- d_k^{\ur,\dagger}(s)-1,\\
		0& \textrm{if~} n\geq 2 (k_\bullet+1)- d_k^{\ur,\dagger}(s).
	\end{cases}
\end{equation}

	(1) Note that by Lemma~\ref{lem:3}, $d_k^{\ur,\dagger}$ is odd dominant, and hence 
	$|d_k^{\ur,\dagger}(s)-d_k^{\ur,\dagger}(s')|\leq 1$. Combined with 
	 \eqref{eq::8}, this implies that if $$m_{n+1}^{(s'),\dagger}(k)-m_{n}^{(s'),\dagger}(k)<m_{n+1}^{(s),\dagger}(k)-m_{n}^{(s),\dagger}(k), $$
	then $n$ must be equal to the boundary point of the first and second cases or the one of the third and fourth cases, i.e., 
	$n=d_k^{\ur,\dagger}(s')-1$ or $2 k_\bullet-d_k^{\ur,\dagger}(s')-1$.
	Note that since $n$ is assumed to be odd, $d_k^{\ur,\dagger}(s')$ is even. Therefore, by the condition~(b) in Definition~\ref{def:1} for $d_k^{\ur,\dagger}$, in both cases, we have $d_k^{\ur,\dagger}(s)=d_k^{\ur,\dagger}(s')$, and hence 
	$m_{n+1}^{(s'),\dagger}(k)-m_{n}^{(s'),\dagger}(k)=m_{n+1}^{(s),\dagger}(k)-m_{n}^{(s),\dagger}(k)$, a contradiction.
	
	(2) Similar to (1), by \eqref{eq::8} and Lemma~\ref{lem:3}, if $$m_{n+1}^{(s'),\dagger}(k)-m_{n}^{(s'),\dagger}(k)>m_{n+1}^{(s),\dagger}(k)-m_{n}^{(s),\dagger}(k),$$
	then 
	$n=d_k^{\ur,\dagger}(s)-1$ or $2 k_\bullet-d_k^{\ur,\dagger}(s)-1$.
	Note that since $n$ is assumed to be even, $d_k^{\ur,\dagger}(s)$ is odd. Therefore,  by the condition~(b) in Definition~\ref{def:1} for $d_k^{\ur,\dagger}$, in both cases, we have $d_k^{\ur,\dagger}(s)=d_k^{\ur,\dagger}(s')$, and hence 
	$m_{n+1}^{(s'),\dagger}(k)-m_{n}^{(s'),\dagger}(k)=m_{n+1}^{(s),\dagger}(k)-m_{n}^{(s),\dagger}(k)$, a contradiction.
\end{proof}

\begin{notation}
	Given any $n\geq 0$ and $w_\star\in \bfm_{\CC_p}$, we denote by $\left(n,h_{n}^{(\us), \dagger}(w_\star)\right)$ the point in $\NP(G^{(\us),\dagger}(w_\star,-))$ over $x=n$.
\end{notation}
\begin{proposition}[The zigzag criterion]\label{critirien}
	Given any $w_\star\in\bfm_{\CC_p}$ and any non-deceasing sequence  $\us=(s_1,\dots,s_u)\in S^u$,  we have $$	\NP(G^{(\us),\dagger}(w_\star,-))=\overunderset{u}{i=1}{\mathlarger{\#}}\NP\left(G^{(s_i),\dagger}(w_\star,-)\right)$$
if and only if for any 
		$1\leq i<j\leq u$, 
		\begin{itemize}
			\item $ h_{n+1}^{(s_j), \dagger}(w_\star)-h_{n}^{(s_j), \dagger}(w_\star)\geq h_{n+1}^{(s_i), \dagger}(w_\star)-h_{n}^{(s_i), \dagger}(w_\star)$ for every odd number $n$; 
			\item  $ h_{n+1}^{(s_j), \dagger}(w_\star)-h_{n}^{(s_j), \dagger}(w_\star)\leq h_{n+1}^{(s_i), \dagger}(w_\star)-h_{n}^{(s_i), \dagger}(w_\star)$ for every even number $n$.
		\end{itemize}
	
\end{proposition}
\begin{proof}
	``$\Longleftarrow$'' 
	We first note that for every $n\geq 0$, by Proposition~\ref{mlemma:Ghost}, we have
	\begin{equation}\label{re:12}
		\v\left(g_{n}^{(\us), \dagger}(w_\star)\right)=\sum\limits_{i=1}^{u}\v\left(g_{n_{u,i}}^{(s_i), \dagger}(w_\star)\right).
	\end{equation}
	Applying the two hypotheses in the second statement on the partitions of $n$ and $n+1$ from Notation~\ref{re:11}, we obtained that 	$$h(n):=\sum\limits_{i=1}^{u}h_{(n+1)_{u,i}}^{(s_i), \dagger}(w_\star)-\sum\limits_{i=1}^{u}h_{n_{u,i}}^{(s_i), \dagger}(w_\star)$$
	is an increasing function of $n\geq 0$; and 
		the set of slopes of $\jing\NP( G^{(s_i),\dagger}(w_\star,-))$ is equal to the multi-set $\{h(n)\;|\;n\geq 0\}.$ Let $\alpha$ be an arbitrary element in $\{h(n)\;|\;n\geq 0\}$. We write $n'$ and $n''$ for the smallest and largest number such that 
		$h(n')=h(n''-1)=\alpha$. Clearly, $n'$ and $n''$ are two vetices of $\NP(G^{(\us),\dagger}(w_\star,-))$, and hence 
	\begin{equation}\label{re::3}	h_{n'}^{(\us),\dagger}(w_\star)=\v\left(g_{n'}^{(\us),\dagger}(w_\star)\right)\quad\textrm{and}\quad h_{n''}^{(\us),\dagger}(w_\star)=\v\left(g_{n''}^{(\us),\dagger}(w_\star)\right).
	\end{equation}
	Moreover, for each $1\leq i\leq u$, we have $$h_{(n')_{u,i}}^{(s_i), \dagger}(w_\star)-h_{(n')_{u,i}-1}^{(s_i), \dagger}(w_\star)<\alpha\leq h_{(n')_{u,i}+1}^{(s_i), \dagger}(w_\star)-h_{(n')_{u,i}}^{(s_i), \dagger}(w_\star),$$
	and hence 
	for each $i\in \{1,\dots,u\}$, the Newton polygon	$\NP\left(G^{(s_i)}(w_\star,-)\right)$ has a vertex over $x=(n')_{u,i}$. This implies that
		\begin{equation}\label{re::4}
				h_{(n')_{u,i}}^{(s_i), \dagger}(w_\star)=\v\left(g_{(n')_{u,i}}^{(s_i), \dagger}(w_\star)\right).
		\end{equation}
		
	Similarly, 
	for each $i\in \{1,\dots,u\}$, the Newton polygon	$\NP\left(G^{(s_i)}(w_\star,-)\right)$ has a vertex over $x=(n'')_{u,i}$, 
	and hence 
		\begin{equation}\label{re::5}
		h_{(n'')_{u,i}}^{(s_i), \dagger}(w_\star)=\v\left(g_{(n'')_{u,i}}^{(s_i), \dagger}(w_\star)\right).
	\end{equation}
	
	Combining \eqref{re::4} with our choices of $n'$ and $n''$, for every $n'+1\leq n\leq n''-1$, we have
	$$\sum\limits_{i=1}^{u}\v\left(g_{n_{u,i}}^{(s_i), \dagger}(w_\star)\right)\geq \sum\limits_{i=1}^{u}h_{n_{u,i}}^{(s_i), \dagger}(w_\star)= \sum\limits_{i=1}^{u}h_{(n')_{u,i}}^{(s_i), \dagger}(w_\star)+\alpha(n-n')=\sum\limits_{i=1}^{u}\v\left(g_{(n')_{u,i}}^{(s_i), \dagger}(w_\star)\right)+\alpha(n-n').$$
	
	On the other hand, combining \eqref{re::4} and \eqref{re::5} with our choices of $n'$ and $n''$, we obtain that 
	for every $n\leq n'-1$,
\begin{equation*}
	\sum\limits_{i=1}^{u}\v\left(g_{(n')_{u,i}}^{(s_i), \dagger}(w_\star)\right)-\sum\limits_{i=1}^{u}\v\left(g_{n_{u,i}}^{(s_i), \dagger}(w_\star)\right)
	\leq \sum\limits_{i=1}^{u}h_{(n')_{u,i}}^{(s_i), \dagger}(w_\star)-\sum\limits_{i=1}^{u}h_{n_{u,i}}^{(s_i), \dagger}(w_\star)< \alpha(n'-n),
\end{equation*}
and 	for every $n\geq n''+1$,
\begin{equation*}
	\sum\limits_{i=1}^{u}\v\left(g_{n_{u,i}}^{(s_i), \dagger}(w_\star)\right)-\sum\limits_{i=1}^{u}\v\left(g_{(n'')_{u,i}}^{(s_i), \dagger}(w_\star)\right)\geq 
	\sum\limits_{i=1}^{u}h_{n_{u,i}}^{(s_i), \dagger}(w_\star)-\sum\limits_{i=1}^{u}h_{(n'')_{u,i}}^{(s_i), \dagger}(w_\star)> \alpha(n-n'').
\end{equation*} 
	The  three chains of inequalities above together imply that 
	the segment connecting 
	$$\left(n',\sum\limits_{i=1}^{u}\v\left(g_{(n')_{u,i}}^{(s_i), \dagger}(w_\star)\right)\right) \textrm{~and~} 	\left(n'',\sum\limits_{i=1}^{u}\v\left(g_{(n'')_{u,i}}^{(s_i), \dagger}(w_\star)\right)\right)$$
		is the segment of slope $\alpha$ in  $\jing\NP\left(G^{(s_i),\dagger}(w_\star,-)\right)$.
	
	On the other hand, by  \eqref{re:12} and \eqref{re::3}, this segment is also the one of slope $\alpha$ in 
	 $\NP(G^{(\us),\dagger}(w_\star,-))$. Since we take $\alpha$ arbitrarily, we obtain
$$	\NP(G^{(\us),\dagger}(w_\star,-))=\jing\NP\left(G^{(s_i),\dagger}(w_\star,-)\right),$$
that completes the proof.

``$\Longrightarrow$'' 
	Suppose that the second statement is false. 
	Let $n$ be the minimal number that fails one of the two items in this statement. By Lemma~\ref{lemma:parity}, we have $n\geq 1$. 
	Without loss of generality, we assume that $n$ is odd, and that there is $1\leq i_0<i_1\leq u$ such that 
	\begin{equation*}
		h_{n+1}^{(s_{i_0}), \dagger}(w_\star)-h_{n}^{(s_{i_0}), \dagger}(w_\star)> h_{n+1}^{(s_{i_1}), \dagger}(w_\star)-h_{n}^{(s_{i_1}), \dagger}(w_\star).
	\end{equation*} 
	Let $\alpha:=\min\limits_{ i_1\leq i\leq u}\{h_{n+1}^{(s_i), \dagger}(w_\star)-h_{n}^{(s_i), \dagger}(w_\star)\}$ and $$i_{\max}:=\max\{i\in \{i_1,\dots,u\}\;|\;h_{n+1}^{(s_i), \dagger}(w_\star)-h_{n}^{(s_i), \dagger}(w_\star)=\alpha\}.$$
	By Lemma~\ref{lemma:parity}, the point $(n, \v(g_{n}^{(s_{i_{\max}}), \dagger}(w_\star))$ must lie above a segment of length $\geq2$ in  $\NP(G^{(s_{i_{\max}}),\dagger}(w_\star,-))$.  We denote by	$n_{i_{\max}}'$ the right endpoint of this segment. Clearly, we have
$
	h_{n}^{(s_{i_{\max}}), \dagger}(w_\star)-h_{n-1}^{(s_{i_{\max}}), \dagger}(w_\star)=h_{n+1}^{(s_{i_{\max}}), \dagger}(w_\star)-h_{n}^{(s_{i_{\max}}), \dagger}(w_\star)= \alpha,
$ and
\begin{equation}\label{eq2}
n_{i_{\max}}'\geq n+1.
\end{equation}

	We generalize the definition of $n_{i}'$ to each $1\leq i\leq u$ by putting
	  $n_{i}'$ be the integer such that $$h_{n_i'}^{(s_i), \dagger}(w_\star)-h_{n_i'-1}^{(s_i), \dagger}(w_\star)\leq \alpha\quad
		\textrm{and}\quad h_{n_i'+1}^{(s_i), \dagger}(w_\star)-h_{n_i'}^{(s_i), \dagger}(w_\star)> \alpha.$$

	From the minimal choice assumption on $n$,  we have 
	$$	h_{n}^{(s_{i_{0}}), \dagger}(w_\star)-h_{n-1}^{(s_{i_{0}}), \dagger}(w_\star)\geq 	h_{n}^{(s_{i_{\max}}), \dagger}(w_\star)-h_{n-1}^{(s_{i_{\max}}), \dagger}(w_\star)=\alpha,$$
	and hence 
	\begin{equation}\label{neq::17}
	n_{i_0'}\leq  n-1.
\end{equation}
	We put 
$N:=\sum\limits_{i=1}^u n_i'$.  
	Combining $|N_{u,i_0}-N_{u,i_{\max}}|\leq 1$ with the inequalities \eqref{eq2} and \eqref{neq::17}, we conclude that the two partitions $(N_{u,1},\dots, N_{u,u})$ 
	and $(n_1',\dots,n_u')$ of $N$ are different. 
	Since $n_i'$ are vertices of $\NP(G^{(s_i),\dagger}(w_\star,-))$ for each $1\leq i\leq u$,
	we obtain that
	\begin{itemize}
		\item 	 the point $\left(N, \sum\limits_{i=1}^uh_{n_i'}^{(s_i), \dagger}(w_\star)\right)$ is a vertex of $\jing\NP(
		G^{(s_i),\dagger}(w_\star,-) ).$ 

\item
The sequence
	 $(n_1',\dots,n_u')$ reaches the minimum of  
	 $	\sum\limits_{i=1}^u h_{n_i'}^{(s_i),\dagger}(w_\star)$ uniquely among all partitions of $N$. 
	 	Hence, 
	 \begin{equation*}
	 	\sum\limits_{i=1}^u	h_{N_{u,i}}^{(s_i),\dagger}(w_\star)>  	\sum\limits_{i=1}^u h_{n_i'}^{(s_i),\dagger}(w_\star).
	 \end{equation*}
	 	\end{itemize}

%
%

	Combined with Proposition~\ref{mlemma:Ghost}, this strict inequality shows that
	\begin{equation}\label{neq::12}
			\v(g^{(\us), \dagger}_{N}(w_\star))=\sum\limits_{i=1}^u\v(g^{(s_i),\dagger}_{N_{u,i}}(w_\star))
			\geq \sum\limits_{i=1}^u h^{(s_i),\dagger}_{N_{u,i}}(w_\star)> 
			\sum\limits_{i=1}^u h_{n_i'}^{(s_i),\dagger}(w_\star),
	\end{equation}
and hence
	$\jing\NP\left(
	G^{(s_i),\dagger}(w_\star,-) \right)$ lies strictly above the point $\left(N, 	\sum\limits_{i=1}^u h_{n_i'}^{(s_i),\dagger}(w_\star)\right)$. By the first item above, this implies that 
	$\NP(G^{(\us),\dagger}(w_\star,-))$ and	$\jing\NP\left(
	G^{(s_i),\dagger}(w_\star,-) \right)$ 
 do not match at $x=N$, and hence
 $$\NP(G^{(\us),\dagger}(w_\star,-))\neq\jing\NP\left(
 G^{(s_i),\dagger}(w_\star,-) \right),$$
	a contradiction.
\end{proof}
We now prove Theorem~\ref{keythm} with assuming the following proposition, whose proof is given at the end of \S\ref{section4.2}.
\begin{proposition}\label{main thm}
	For any sequence  $\us\in S^u$ of finite length, if there exist $i,j\in \{1,\dots,u\}$ such that 
	$s_i\neq s_j$, then there exists  $w_\star\in \bfm_{\CC_p}$ such that 
	$$\jing\NP\left( G_k^{(s_i),\dagger}(w_\star,-)\right)\neq \NP\left(G_k^{(\us),\dagger}(w_\star,-)\right).$$
\end{proposition}
\begin{proof}[Proof of Theorem~\ref{keythm}]
	By Proposition~\ref{re:4}, we can assume that for every $i\in \{1,\dots, u\}$, $\tH_i=\tH(s_i)$ for some $s_i\in \{0,\dots, p-2\}$.  Combined with Lemma~\ref{re:14}, this simplification reduces Theorem~\ref{keythm} to proving that  for any sequence
	$\us\in \{0,\dots,p-2\}^u$ such that $\iota(s_i)$ (see Definition~\ref{re::1}(1)) is generic for every $1\leq i\leq u$, the following statements are equivalent. 
		\begin{enumerate}
		\item[$(1')$] For every $w_\star\in \bfm_{\CC_p}$, we have 
		\begin{equation*}
			\NP\left(G^{(\us)}(w_\star,-)\right)=\jing \NP\left(G^{(s_i)}(w_\star,-)\right).
		\end{equation*}
		\item[$(2')$] For every distinct $i, j\in \{1,\dots, u\}$, we have $s_i=s_j$ or
		$s_i+s_j\equiv k_0-1\pmod{p-1}$.
	\end{enumerate}	
By Lemmas~\ref{re:5} and \ref{lem:2},  the equality in $(1')$ can be replaced by  
	 \begin{equation}\label{neweq4}
	\NP\left(G^{(\us),\dagger}(w_\star,-)\right)= \jing\NP\left(G^{(s_i),\dagger}(w_\star,-)\right).
\end{equation}
Combined with Lemma~\ref{re:15}, this theorem can further reduced to proving that  the following statements are equivalent for every $\us\in S^u$.
\begin{enumerate}
	\item[$(1'')$] The relation~\eqref{neweq4} holds for every $w_\star\in \bfm_{\CC_p}$.
	\item[$(2'')$] The indices $s_i$'s are identical for all $i\in \{1,\dots, u\}$.
\end{enumerate}
It follows from Proposition~\ref{main thm} that $(1'')$ implies $(2'')$, and from Proposition~\ref{critirien}, its reverse statement. 
\end{proof}

\section{Proof of Proposition~\ref{main thm}}

\subsection{Some general results.}
As in  \cite[Notation~5.1]{xiao}.
\begin{notation}
	For any $k\in \calK$ and any $s\in \{0,\dots,p-2\}$, we write 
	\begin{gather}	d_k^\new(s):=d_k^\Iw(s)-2d_k^\ur(s),\notag
\\
\label{neq::20}
		{\Delta}_{k, \ell}^{(s)}:=v_{p}\left(g_{\frac{1}{2} d_{k}^{\Iw}+\ell, \hat{k}}^{(s)}\left(w_{k}\right)\right)-\frac{k-2}{2} \ell, 
	\end{gather}
	for $\ell=-\frac{1}{2} d_{k}^\new(s),-\frac{1}{2} d_{k}^\new(s)+1, \ldots, \frac{1}{2} d_{k}^\new(s)$,
where	$g_{\frac{1}{2} d_{k}^{\Iw}+\ell, \hat{k}}^{(s)}\left(w\right)$ is the polynomial of $w$  constructed by removing $(w-w_k)$-factors in $g_{\frac{1}{2} d_{k}^{\Iw}+\ell}^{(s)}\left(w\right)$. 
	By \cite[Proposition~4.18(4)]{xiao}, we have
	\begin{equation}\label{neq::19}
		{\Delta}_{k, \ell}^{(s)}={\Delta}_{k,-\ell}^{(s)} \quad \text { for every } \ell=-\frac{1}{2} d_{k}^{\text {new }}, -\frac{1}{2} d_{k}^{\text {new }}+1, \ldots, \frac{1}{2} d_{k}^{\text {new }}.
	\end{equation}
	
	We denote by $\Delta_{k}^{(s)}$ the 
	 lower convex hull  of the set of points $$\left\{(\ell, \overline \Delta_{k, \ell}^{(s)})\;\Big|\; -\frac{1}{2}d_k^\new(s)\leq \ell \leq\frac{1}{2} d_k^\new(s)
	 \right\},$$
	and by $\left(\ell, \OD_{k, \ell}^{(s)}\right)$ the corresponding points on this convex hull. 
	Note that, for simplicity of notation, we use  	${\Delta}_{k, \ell}^{(s)}$ and $\OD_{k, \ell}^{(s)}$ to denote 	${\Delta'}_{k, \ell}$ and 	${\Delta}_{k, \ell}$ in \cite{xiao}. 
\end{notation}

\begin{notation}\label{notation:1}
Recall that we define $\beta^{(s)}_{n}$ in Notation~\ref{notation:2}. 	For any $k\in \calK$, $s\in \{0,1,\dots,p-2\}$ and $\ell\in \ZZ$, we write
	\begin{itemize}
		\item $\theta^{(s)}_{k,\ell}:=\beta^{(s)}_{k_\bullet-\delta_{s}-\ell}-\beta^{(s)}_{k_\bullet+1-\delta_{s}-\ell}+\frac{p+1}{2}$, which is equal to either $a_{s} + 2$ or $p-1-a_{s}$; 
		\item $A^{(s)}_{k,\ell}:=\frac{1}{2}\left(k-2-(p+1)\ell+\theta_{k,\ell+1}^{(s)}\right)$;
		\item $B^{(s)}_{k,\ell}:=\frac{1}{2}\left(k-2+(p+1)\ell-\theta_{k,\ell+1}^{(s)}\right)$.

	\end{itemize}

\end{notation}

\begin{lemma}\label{lem:8}
	For any $k\in \calK$, $s\in S$ and $\ell\in \ZZ$, 
	\begin{enumerate}
		\item if $k_\bullet+1-\ell$ is even, then
		\begin{equation*}
			\theta^{(s)}_{k,\ell}=2s+2-k_0;
		\end{equation*}
		\item if $k_\bullet+1-\ell$ is odd, then
		\begin{equation*}
			\theta^{(s)}_{k,\ell}=p-1-2s+k_0;
		\end{equation*}
		\item $3\leq \theta^{(s)}_{k,\ell}\leq p-2;$ and
		\item as a function of $s\in S$, $B^{(s)}_{k,\ell}$ is increasing if $k_\bullet+1-\ell$ is even, while decreasing if $k_\bullet+1-\ell$ is odd. 
	\end{enumerate}
	
\end{lemma}

\begin{proof}	
	(1)	By Lemma~\ref{lem:1}, for $\lceil\frac{k_0+1}{2}\rceil\leq s\leq k_0-2$, we have $\delta_{s}=0$, and hence $$\theta^{(s)}_{k,\ell}=\beta^{(s)}_{1}-\beta^{(s)}_{0}+\frac{p+1}{2}=t_{2}^{(s)}-t_{1}^{(s)}=2s+2-k_0.$$
	
	Similarly, for $k_0-1\leq s\leq \lfloor\frac{k_0+p-4}{2}\rfloor$, we have
	$\delta_{s}=1$, and hence $$\theta^{(s)}_{k,\ell}=\beta^{(s)}_{0}-\beta^{(s)}_{1}+\frac{p+1}{2}=p+1+t_{1}^{(s)}-t_{2}^{(s)}=2s+2-k_0.$$
	

	(2) From (1), we know that  $\theta^{(s)}_{k,\ell+1}=2s+2-k_0$. Combined with 
	\begin{equation}\label{eq:12}
		\theta^{(s)}_{k,\ell+1}+\theta^{(s)}_{k,\ell}=p+1,
	\end{equation} 
	this equality completes the proof. 
	
	(3) If $k_\bullet+1-\ell$ is even, by (1), we have $\theta^{(s)}_{k,\ell}=2s+2-k_0.$
	Combined with
	\begin{equation}\label{neq::23}
		3\leq 2\left\lceil\frac{k_0+1}{2}\right\rceil+2-k_0 \leq 2s+2-k_0\leq 2\left\lfloor\frac{k_0+p-4}{2}\right\rfloor+2-k_0\leq p-2,
	\end{equation}
this equality proves this case. 

If $k_\bullet+1-\ell$ is odd, note that we have proved $3\leq \theta^{(s)}_{k,\ell+1}\leq p-2.$
Combined with \eqref{eq:12},
 this chain of inequalities completes the proof.
 
(4) It is a direct consequence of (1) and (2). 
\end{proof}

\begin{notation}\label{aa:10}
	For a positive integer $m$, let $\operatorname{Dig}(m)$ denote the sum of all digits in the $p$-based expression of $m$. Then the sum of valuations of consecutive integers in $\left(m_{1}, m_{2}\right]$ with $m_{2}>m_{1}>0$ is
	\begin{equation}\label{bb:1}
\sum\limits_{m_{1}<i \leq m_{2}} v_{p}(i)=\frac{\left(m_{2}-\operatorname{Dig}\left(m_{2}\right)\right)-\left(m_{1}-\operatorname{Dig}\left(m_{1}\right)\right)}{p-1} .
	\end{equation}
\end{notation}

\begin{lemma}\label{lem:6}
For any $k\in \calK$, $s\in S$ and $1\leq \ell\leq\frac{1}{2}d^{\new}_k(s)$, we have
	$$
	{\Delta}_{k, \ell}^{(s)}- {\Delta}_{k, \ell-1}^{(s)}=
	\frac{(p-1)(\ell-1)+\theta_{k,\ell}^{(s)}}{2}+\frac{\theta_{k,\ell}^{(s)}+\Dig\left(A^{(s)}_{k,\ell}\right)+2 \Dig(\ell-1)-\Dig\left(B^{(s)}_{k,\ell}\right)}{p-1}.
	$$

\end{lemma}
\begin{proof}
	Let $$n:=\frac{1}{2} d_{k}^{\Iw}(s)-\ell=k_\bullet+1-\delta_{s}-\ell\quad\textrm{and}\quad\eta:=\frac{p-1}{2} k_\bullet-\frac{p+1}{2} \delta_{s}+\beta^{(s)}_{n}-1.$$
	Then we have 
	$$\theta^{(s)}_{k,\ell}=\beta^{(s)}_{n-1}-\beta^{(s)}_{n}+\frac{p+1}{2}.$$
	
	Combined with \cite[(5.3.6)]{xiao}, this equality implies that
	\begin{multline}\label{eq:10}
		{\Delta}_{k, \ell}^{(s)}- {\Delta}_{k, \ell-1}^{(s)}\\=
		\frac{(p-1)(\ell-1)+\theta_{k,\ell}^{(s)}}{2}+\frac{\theta_{k,\ell}^{(s)}+\Dig\left(\eta-\frac{p+1}{2}(\ell-1)\right)+2 \Dig(\ell-1)-\Dig\left(\eta+\theta_{k,\ell}^{(s)}+\frac{p+1}{2}(\ell-1)\right)}{p-1}.
	\end{multline}

	On the other hand, by Lemma~\ref{lem:1},  we have 
	\begin{equation*}
		-(p+1) \delta_{s}+\beta^{(s)}_{n}+\beta^{(s)}_{n+1}=k_0-\frac{p+1}{2}.
	\end{equation*}
	Combined with $\beta_{n-1}^{(s)}=\beta_{n+1}^{(s)}$, this equality shows that
	\begin{align*}
		\eta-\frac{p+1}{2}(\ell-1)	=&\frac{p-1}{2} k_\bullet-\frac{p+1}{2} \delta_{s}+\beta^{(s)}_{n}-1-\frac{p+1}{2}(\ell-1)\\=&
		\frac{1}{2}\left(k-k_0-(p+1)\delta_{s}+\frac{p+1}{2}+2\beta^{(s)}_{n}-2-(p+1)\ell+\frac{p+1}{2}\right)
		\\=&
		\frac{1}{2}\left(k-\beta^{(s)}_{n}-\beta^{(s)}_{n+1}+2\beta^{(s)}_{n}-2-(p+1)\ell+\frac{p+1}{2}\right)
		\\=& \frac{1}{2}\left(k+\theta^{(s)}_{k,\ell+1}-2-(p+1)\ell\right)\\
		=&A^{(s)}_{k,\ell}.
		\end{align*}
	Together with \eqref{eq:12}, this chain of equalities implies that \[\eta+\theta_{k,\ell}^{(s)}+\frac{p+1}{2}(\ell-1)=A^{(s)}_{k,\ell}+\theta_{k,\ell}^{(s)}+\frac{p+1}{2}(2\ell-1)=B^{(s)}_{k,\ell}.\qedhere\] 
\end{proof}

%
%
%

\begin{lemma}\label{shift}
	For any $k\in \calK$ and $s\in \{0,\dots,p-2\}$, we have
	\begin{equation*}
		B^{(s)}_{k,\ell+r+1}-A^{(s)}_{k,\ell-r}=B^{(s)}_{k,\ell-r}-	A^{(s)}_{k,\ell+r+1}=(p+1)\ell.
	\end{equation*}
\end{lemma}
\begin{proof}
	It follows that	
	\begin{align*}
		B^{(s)}_{k,\ell+r+1}-A^{(s)}_{k,\ell-r}
		=&\frac{1}{2}\left(k-2+(p+1)(\ell+r+1)-\theta_{k,\ell+r+2}^{(s)}\right)-\frac{1}{2}\left(k-2-(p+1)(\ell-r)+\theta_{k,\ell-r+1}^{(s)}\right)\\
		=&(p+1)\ell
	\end{align*}
	and 
	\begin{align*}		
		B^{(s)}_{k,\ell-r}-	A^{(s)}_{k,\ell+r+1}
		=&\frac{1}{2}\left(k-2+(p+1)(\ell-r)-\theta_{k,\ell+r+1}^{(s)}\right)-\frac{1}{2}\left(k-2-(p+1)(\ell+r+1)+\theta_{k,\ell-r+2}^{(s)}\right)\\=&(p+1)\ell.\qedhere
	\end{align*}
	
\end{proof}

\begin{notation}\label{notation:s}
	We write \[s_{k,\ell}:=\begin{cases}
		\lceil  \frac{k_0+1}{2}\rceil&\textrm{if~}  k_\bullet+1-\ell\textrm{~is odd},
		\\
		\lfloor  \frac{k_0+p-4}{2}\rfloor&\textrm{if~} k_\bullet+1-\ell \textrm{~is even},
	\end{cases}\]
	and \begin{multline*}
		2P_{k,\ell}:=	\frac{(p-1)(2\ell-1)+p+1}{2}
		+\frac{p+1+2(\Dig(\ell)+\Dig(\ell-1))}{p-1}\\
		-
		\frac{\Dig\left(B^{(s_{k,\ell})}_{k,\ell+1}\right)-\Dig\left(A^{(s_{k,\ell})}_{k,\ell}\right)+\Dig\left(B^{(s_{k,\ell})}_{k,\ell}\right)-\Dig\left(A^{(s_{k,\ell})}_{k,\ell+1}\right)}{p-1}.
	\end{multline*}
\end{notation}

\begin{lemma}\label{lem:7}
For any $k\in \calK$, $s\in S$, $1\leq \ell\leq\frac{1}{3}d^{\new}_k(s)$ and $1\leq r\leq p^{\v(\ell)-1}$, if there does not exist $p^{\v(\ell)}$-divisible integer in 
	$(B^{(s)}_{k,\ell-r+1}, B^{(s)}_{k,\ell+r}] $,
	then
	$$	\Delta^{(s)}_{k, \ell-r+1}-\Delta^{(s)}_{k, \ell-r}+\Delta^{(s)}_{k, \ell+r}-\Delta^{(s)}_{k, \ell+r-1}\\
	=2P_{k,\ell}. $$
\end{lemma}
\begin{proof}
We first note that $$\ell+r\leq \frac{p+1}{p}\ell\leq \frac{p+1}{3p}d^{\new}_k(s)\leq\frac{1}{2}d^{\new}_k(s).$$

	By Lemma~\ref{lem:6}, we have
	\begin{multline*}		
			{\Delta}_{k, \ell-r+1}^{(s)}- {\Delta}_{k, \ell-r}^{(s)}+\Delta^{(s)}_{k, \ell+r}-\Delta^{(s)}_{k, \ell+r-1}
		\\
\begin{aligned}
	=&\frac{(p-1)(\ell-r)+\theta_{k,\ell-r+1}^{(s)}}{2}+\frac{\theta_{k,\ell-r+1}^{(s)}+\Dig\left(A^{(s)}_{k,\ell-r+1}\right)+2 \Dig(\ell-r)-\Dig\left(B^{(s)}_{k,\ell-r+1}\right)}{p-1}
		\\	
		+&\frac{(p-1)(\ell+r-1)+\theta_{k,\ell+r}^{(s)}}{2}+\frac{\theta_{k,\ell+r}^{(s)}+\Dig\left(A^{(s)}_{k,\ell+r}\right)+2 \Dig(\ell+r-1)-\Dig\left(B^{(s)}_{k,\ell+r}\right)}{p-1}.
\end{aligned}
	\end{multline*}	
Combining the hypothesis in this lemma with Lemma~\ref{lem:8}(4), we obtain that $(B^{(s_{k,\ell})}_{k,\ell}, B^{(s)}_{k,\ell+r}]\subset(B^{(s)}_{k,\ell-r+1}, B^{(s)}_{k,\ell+r}]$
does not contain $p^{\vl}$-divisible integers.
	Combined with Lemma~\ref{shift} that the interval 
	$(A^{(s_{k,\ell})}_{k,\ell+1}, A^{(s)}_{k,\ell-r+1}]$ is a shift of 
$	(B^{(s_{k,\ell})}_{k,\ell}, B^{(s)}_{k,\ell+r}]$ to the left by $(p+1)\ell$, this implies that
	every $p$-divisible integer in $(A^{(s_{k,\ell})}_{k,\ell+1}, A^{(s)}_{k,\ell-r+1}]$ corresponds to one in $(B^{(s_{k,\ell})}_{k,\ell}, B^{(s)}_{k,\ell+r}]$ with the same $p$-adic valuation. 
	Therefore, by \eqref{bb:1}, we have 
	$$\Dig(B^{(s)}_{k,\ell+r})-\Dig(B^{(s_{k,\ell})}_{k,\ell})=\Dig(A^{(s)}_{k,\ell-r+1})-\Dig(A^{(s_{k,\ell}}_{k,\ell+1}).$$
	Similarly, we have 
		$$\Dig(B^{(s_{k,\ell})}_{k,\ell+1})-\Dig(B^{(s)}_{k,\ell-r+1})=\Dig(A^{(s_{k,\ell})}_{k,\ell})-\Dig(A^{(s)}_{k,\ell+r}).$$

	On the other hand, by our hypothesis $r\leq p^{\vl-1}$ and the symmetry of $p$-divisible integers in $[\ell-r+1,\ell+r-1]$ about $x=\ell$, we have $$\Dig(\ell+r-1)-\Dig(\ell)=\Dig(\ell-1)-\Dig(\ell-r).$$

	Combining all these equalities above, we complete the proof.
\end{proof}

\begin{corollary}\label{corollary:slope}
	For any $k\in \calK$, $s\in S$, $1\leq \ell\leq\frac{1}{3}d^{\new}_k(s)$ and $1\leq r\leq p^{\v(\ell)-1}$ if there does not exist $p^{\v(\ell)}$-divisible integer in 
	$(B^{(s)}_{k,\ell-r+1}, B^{(s)}_{k,\ell+r}] $, then
	 the segment connecting $(\ell-r, \Delta^{(s)}_{k,\ell-r})$ and  $(\ell+r, \Delta^{(s)}_{k,\ell+r})$ is of slope $P_{k,\ell}$.
\end{corollary}

\begin{proof}
	Note that $$\frac{\Delta^{(s)}_{k,\ell+r}-\Delta^{(s)}_{k,\ell-r}}{2r}=\sum\limits_{r'=1}^r \frac{	\Delta^{(s)}_{k, \ell-r'+1}-\Delta^{(s)}_{k, \ell-r'}+\Delta^{(s)}_{k, \ell+r'}-\Delta^{(s)}_{k, \ell+r'-1}}{2r}.$$
	Combined with Lemma~\ref{lem:7}, this equality shows that  $\frac{\Delta^{(s)}_{k,\ell+r}-\Delta^{(s)}_{k,\ell-r}}{2r}=P_{k,\ell}$, which completes the proof.
\end{proof}

\begin{lemma}\label{lemma:r_1r_2}
	For any $k\in \calK$, $s\in S$ and $1\leq \ell_1<\ell_2\leq d_{k}^\new(s) $, we have 
	\begin{multline*}
		\Delta^{(s)}_{k, \ell_2}-\Delta^{(s)}_{k, \ell_2-1}-\left(\Delta^{(s)}_{k, \ell_1}-\Delta^{(s)}_{k, \ell_1-1}\right)\\=	\frac{(p-3)(\ell_2-\ell_1)+\theta_{k,\ell_2}^{(s)}-\theta_{k,\ell_1}^{(s)}}{2}+\left(\sum\limits_{i=B^{(s)}_{k,\ell_1}+1}^{B^{(s)}_{k,\ell_2}}+\sum\limits_{i=A^{(s)}_{k,\ell_2}+1}^{A^{(s)}_{k,\ell_1}}-2\sum\limits_{i=\ell_1}^{\ell_2-1}\right) \v(i).
	\end{multline*}
\end{lemma}
\begin{proof}
	We first note that 
	\begin{multline*}
		\begin{aligned}
				\Dig\left(B^{(s)}_{k,\ell_2}\right)-	\Dig\left(B^{(s)}_{k,\ell_1}\right)	= &	B^{(s)}_{k,\ell_2}-	B^{(s)}_{k,\ell_1}-(p-1)\sum\limits_{i=B^{(s)}_{k,\ell_1}+1}^{B^{(s)}_{k,\ell_2}}\v(i)\\
			=&\frac{(\ell_2-\ell_1)(p+1)-\theta^{(s)}_{k,\ell_2+1}+\theta^{(s)}_{k,\ell_1+1}}{2}-(p-1)\sum\limits_{i=B^{(s)}_{k,\ell_1}+1}^{B^{(s)}_{k,\ell_2}}\v(i),
		\end{aligned}
	\end{multline*}
	and similarly that 
	\begin{equation*}
			\Dig\left(A^{(s)}_{k,\ell_1}\right)-	\Dig\left(A^{(s)}_{k,\ell_2}\right)
		=
		\frac{(\ell_2-\ell_1)(p+1)-\theta^{(s)}_{k,\ell_2+1}+\theta^{(s)}_{k,\ell_1+1}}{2}-(p-1)\sum\limits_{i=A^{(s)}_{k,\ell_2}+1}^{A^{(s)}_{k,\ell_1}}\v(i).
	\end{equation*}
	
	Plugging them into the equality in Lemma~\ref{lem:6}, we have
\begin{multline*}
		\Delta^{(s)}_{k, \ell_2}-\Delta^{(s)}_{k, \ell_2-1}-\left(\Delta^{(s)}_{k, \ell_1}-\Delta^{(s)}_{k, \ell_1-1}\right)\\
		\begin{aligned}	
		=&
		\frac{(p-1)(\ell_2-1)+\theta_{k,\ell_2}^{(s)}}{2}+\frac{\theta_{k,\ell_2}^{(s)}+\Dig\left(A^{(s)}_{k,\ell_2}\right)+2 \Dig(\ell_2-1)-\Dig\left(B^{(s)}_{k,\ell_2}\right)}{p-1}\\
		-&	\frac{(p-1)(\ell_1-1)+\theta_{k,\ell_1}^{(s)}}{2}-\frac{\theta_{k,\ell_1}^{(s)}+\Dig\left(A^{(s)}_{k,\ell_1}\right)+2 \Dig(\ell_1-1)-\Dig\left(B^{(s)}_{k,\ell_1}\right)}{p-1}\\
		=
		&
		\frac{(p-1)(\ell_2-\ell_1)+\theta_{k,\ell_2}^{(s)}-\theta_{k,\ell_1}^{(s)}}{2}+\frac{\theta_{k,\ell_2}^{(s)}-\theta_{k,\ell_1}^{(s)}+2 \Dig(\ell_2-1)-2 \Dig(\ell_1-1)}{p-1}\\
		-&	\frac{	(\ell_2-\ell_1)(p+1)-\theta^{(s)}_{k,\ell_2+1}+\theta^{(s)}_{k,\ell_1+1}}{p-1}+\sum\limits_{i=B^{(s)}_{k,\ell_1}+1}^{B^{(s)}_{k,\ell_2}}\v(i)+\sum\limits_{i=A^{(s)}_{k,\ell_2}+1}^{A^{(s)}_{k,\ell_1}}\v(i)\\
		=&
		\frac{(p-3)(\ell_2-\ell_1)+\theta_{k,\ell_2}^{(s)}-\theta_{k,\ell_1}^{(s)}}{2}+\left(\sum\limits_{i=B^{(s)}_{k,\ell_1}+1}^{B^{(s)}_{k,\ell_2}}+\sum\limits_{i=A^{(s)}_{k,\ell_2}+1}^{A^{(s)}_{k,\ell_1}}-2\sum\limits_{i=\ell_1}^{\ell_2-1}\right) \v(i).
	\end{aligned}
\end{multline*}
This completes the proof.
\end{proof}

\begin{corollary}\label{corollary:r_1r_2}
	For any $k\in \calK$, $s\in S$ and $1\leq \ell_1<\ell_2\leq d_{k}^\new(s) $, we have 
	\begin{equation*}
		\Delta^{(s)}_{k, \ell_2}-\Delta^{(s)}_{k, \ell_2-1}-\left(\Delta^{(s)}_{k, \ell_1}-\Delta^{(s)}_{k, \ell_1-1}\right)
		\geq 	\frac{(p-3)(\ell_2-\ell_1-1)}{2}-2\sum\limits_{i=\ell_1}^{\ell_2-1} \v(i)+1.
	\end{equation*}
\end{corollary}

\begin{proof}
	
By Lemma~\ref{lem:8}(3), we have 
$$\theta_{k,\ell_2}^{(s)}-\theta_{k,\ell_1}^{(s)}\geq  5-p.$$
Plugging it into Lemma~\ref{lemma:r_1r_2}, we complete the proof.
\end{proof}

\subsection{Concrete construction.}\label{section4.2}
For the rest of this section, we fix $s_1<s_2\in S$. By symmetry, we may and will assume that $s_1\neq k_0-1$.	We fix $\ell:= p^{p^{p+2}}$ and $r_0:=p^{2p}$. Let $s_0:=\left\lceil\frac{k_0+1}{2}\right\rceil$ be the smallest term in $S$. 
Recall that we assume $p\geq 7$. 

\begin{notation}
\begin{enumerate}
	\item For every $k\in \calK$, $s\in S$ and $r\in \ZZ_{\geq 0}$ such that $\ell+r\leq \frac{1}{2}d_k^\new(s)$, we write \begin{equation}\label{aa:2}
		F^{(s)}_k(r):=\Delta^{(s)}_{k, \ell+r}-\Delta^{(s)}_{k, \ell+r-1}-	(\Delta^{(s)}_{k, \ell-r+1}-\Delta^{(s)}_{k, \ell-r}).
	\end{equation}
	\item 	For every integer $w\geq 1$, we put \begin{gather*}
		k_w:=2\ell(p-1)+2(p^w-1)(k_0-1-s_1)+k_0-p+1\in \calK \\ 	Q_w:=\frac{3p-1}{2}\ell +p^w(k_0-1-s_1).
	\end{gather*}
\end{enumerate}	
\end{notation}

\begin{lemma}\label{lem:9}
	For every integer $w\geq 1$, $s\in S$ and $r\geq 1$, we have
	\begin{enumerate}
		\item 
		$
			\theta_{k_w,\ell+r}^{(s)}=\begin{cases}
				p-1-2s+k_0& \textrm{if~} 2\;|\;r,\\
				2s+2-k_0& \textrm{if~} 2\nmid r.\\
			\end{cases}
		$	
		\item $B^{(s)}_{k_w,\ell+r}=			Q_w+	\frac{r(p+1)}{2}+	\begin{cases}
			s_1-s-	\frac{p+1}{2}	& \textrm{if~} 2\;|\;r,\\
			s_1+s-k_0-p& \textrm{if~} 2\nmid r.
		\end{cases}	$
	\end{enumerate}

\end{lemma}
\begin{proof}
(1)	Note that for any $w\geq 1$, the parity of 
\begin{equation*}
	{k_w}_\bullet+1-(r+\ell)=2\ell+2(k_0-1-s_1)\frac{p^w-1}{p-1}-(r+\ell)
\end{equation*} is opposite to $r$. Hence, by Lemma~\ref{lem:8}, we have
	\begin{equation*}
		\theta_{k_w,\ell+r}^{(s)}=\begin{cases}
			p-1-2s+k_0& \textrm{if~} 2\;|\;r,\\
			2s+2-k_0& \textrm{if~} 2\nmid r.\\
		\end{cases}
	\end{equation*}
	
(2)	It follows that \begin{align*}
		B^{(s)}_{k_w,\ell+r}=
		&\frac{k_w-2+(\ell+r)(p+1)-\theta^{(s)}_{k,\ell+r+1}}{2}\\
		=&Q_w+\frac{2s_1-(k_0-1)+r(p+1)-\theta^{(s)}_{k,\ell+r+1}}{2}
		\\
		=&
		Q_w+	\frac{r(p+1)}{2}+	\begin{cases}
	s_1-s-	\frac{p+1}{2}	& \textrm{if~} 2\;|\;r,\\
s_1+s-k_0-p& \textrm{if~} 2\nmid r.
		\end{cases}\qedhere
		\end{align*}
\end{proof}

\begin{lemma}\label{mlem:1}\hspace{2em}
\begin{enumerate}
	\item There exists unique $2p+2<w_{0}<p^{p+2}$ such that 
	$$
	-1\leq F_{k_{w_0}}^{(s_0)}(r_0-1)\leq 0< F_{k_{w_0}}^{(s_0)}(r_0).$$

	\item 	For every $s\in S\backslash \{s_0\}$, we have 
	$$
	F_{k_{w_0}}^{(s)}(r_0-1)< 0< F_{k_{w_0}}^{(s)}(r_0).$$
	
\end{enumerate}
%

\end{lemma}

\begin{proof}
By Lemmas~\ref{shift} and \ref{lemma:r_1r_2}, for any $1\leq r\leq r_0$ and $s\in S $, we have
\begin{align}\label{eq::1}
		F^{(s)}_{k_{w}}(r)
		=&\frac{(p-3)(2r-1)+\theta_{\ell+r}^{(s)}-\theta_{\ell+r-1}^{(s)}}{2}+\left(\sum\limits_{i=B^{(s)}_{k_{w},\ell-r+1}+1}^{B^{(s)}_{k_{w},\ell+r}}+\sum\limits_{i=A^{(s)}_{k_w,\ell+r}+1}^{A^{(s)}_{k_{w},\ell-r+1}}-2\sum\limits_{i=\ell-r+1}^{\ell+r-1}\right)\v(i)\\
	\notag	=&\frac{(p-3)(2r-1)+\theta_{\ell+r}^{(s)}-\theta_{\ell+r-1}^{(s)}}{2}+2\left( \sum\limits_{i=B^{(s)}_{k_{w},\ell-r+1}+1}^{B^{(s)}_{k_{w},\ell+r}}-\sum\limits_{i=\ell-r+1}^{\ell+r-1}\right)\v(i).
\end{align}

By Lemma~\ref{lem:9}(1) and $r_0=p^{2p}$, we have
	\begin{equation}\label{neq::8}
	\theta_{k_{w},\ell+r_0}^{(s_0)}=	2s+2-k_0.
\end{equation}
Combined with Lemma~\ref{lem:8}(3), this inequality implies that
\begin{multline}\label{eq::26}
F^{(s_0)}_{k_{w}}(r_0)-	F^{(s_0)}_{k_{w}}(r_0-1)\\
	\begin{aligned}
	=& \frac{2(p-3)-2(p+1)+4\theta_{k_{w},\ell+r_0}^{(s_0)}}{2}+2\left(\sum\limits_{i=B^{(s_0)}_{k_{w},\ell+r_0-1}+1}^{B^{(s_0)}_{k_{w},\ell+r_0}}+\sum\limits_{i=B^{(s_0)}_{k_{w},\ell-r_0+1}+1}^{B^{(s_0)}_{k_{w},\ell-r_0+2}}\right)\v(i)
	\\&\qquad\qquad\qquad\qquad\qquad\qquad\qquad\qquad-2\v(\ell-r_0+1)-2\v(\ell+r_0-1)
	\\\geq& \frac{2(p-3)-2(p+1)+4\theta_{k_{w},\ell+r_0}^{(s_0)}}{2}\\
	\geq &\frac{-8+4\times 3}{2}\geq 2.
\end{aligned}
\end{multline}

%
Now we only discuss the property of $k_w$ for $w$ such that $2p+2<w<p^{p+2}$.
Note that by Lemma~\ref{lem:9}, it follows that $Q_w\in [B^{(s_0)}_{k_{w},\ell-(r_0-1)+1}+1, B^{(s_0)}_{k_{w},\ell+r_0-1}]$. Since 
the length of the interval $[B^{(s_0)}_{k_{w}+1,\ell-(r_0-1)+1}, B^{(s_0)}_{k_{w},\ell+r_0-1}]$ is equal to 
\begin{equation}\label{re:16}
	B^{(s_0)}_{k_{w},\ell+r_0-1}-B^{(s_0)}_{k_{w},\ell-r_0+2}-1=(r_0-2)(p+1)-2s_0+k_0+p-1<r_0(p+1)<p^w,
\end{equation}
$Q_w$ is the unique $p^w$-divisible integer in $[B^{(s_0)}_{k_{w},\ell-r_0+2}+1, B^{(s_0)}_{k_{w},\ell+r_0-1}]$, whose $p$-adic valuation is exactly $w$. 
This implies that 
$\left(	\sum\limits_{i=B^{(s_0)}_{k_{w},\ell-r_0+2}+1}^{B^{(s_0)}_{k_{w},\ell+r_0-1}}\v(i)-w\right)$ is constant for every $2p+2<w<p^{p+2}$.
Combining it with the expression of $F_{k_{w}}^{(s_0)}(r_0-1)$ in \eqref{eq::1}, we know that $h:=F_{k_{w}}^{(s_0)}(r_0-1)-2w$ is also independent to $w$.  
If 
\begin{equation}\label{neq::7}
	-2p^{p+2}+2<h<-4p-6,
\end{equation}
 then taking $w_0:=\left\lfloor\frac{-h}{2}\right\rfloor$, we have 
$$ F_{k_{w_0}}^{(s_0)}(r_0-1)=h+2\left\lfloor\frac{-h}{2}\right\rfloor\in [-1,0]\quad\textrm{and}\quad w_0\in [2p+3, p^{p+2}-1].$$
Combined with \eqref{eq::26}, this $w_0$ satisfies the both conditions in (1). Clearly, every $w(\neq w_0)\in (2p+2,p^{p+2})$ satisfies $$|F_{k_{w}}^{(s_0)}(r_0-1)-F_{k_{w}}^{(s_0)}(r_0-1)|\geq 2.$$
Hence, such $w_0$ if exists, then is unique.
Now we are left to prove \eqref{neq::7}.
%
Applying on \eqref{eq::1} that 
\begin{gather*}
	 \sum\limits_{i=\ell-r_0+2}^{\ell+r_0-2}\v(i)=p^{p+2}+2\sum\limits_{i=1}^{r_0-2}\v(i)\leq  p^{p+2}+2\sum\limits_{i=1}^{2p}\frac{r_0-2}{p^i}\leq p^{p+2}+\frac{2r_0}{p-1}, \\
	 \frac{(p-3)(2r_0-3)+\theta_{k_{w},\ell+r_0-1}^{(s_0)}-\theta_{k_{w},\ell+r_0-2}^{(s_0)}}{2}\geq \frac{2r_0-3}{2},\\
\textrm{and}\quad	 \sum\limits_{i=B^{(s_0)}_{k_{w},\ell-r_0+2}+1}^{B^{(s_0)}_{k_{w},\ell+r_0-1}}\v(i)-w\geq 0, 
\end{gather*}
we complete the proof of the left hand side of \eqref{neq::7}.

On the other hand, by \eqref{re:16} and the fact that $Q_w$ is the unique $r_0$-divisible integer in $[B^{(s_0)}_{k_{w},\ell-r_0+1}+1,B^{(s_0)}_{k_{w},\ell+r_0}]$, we have the following estimate:
\begin{equation*}
	\sum\limits_{i=B^{(s_0)}_{k_{w},\ell-r_0+2}+1}^{B^{(s_0)}_{k_{w},\ell+r_0-1}}\v(i)\leq \sum\limits_{i=Q_w-r_0(p+1)}^{Q_w+r_0(p+1)}\v(i)=w+2\sum\limits_{j=1}^{p+1} \left\lfloor \frac{r_0(p+1)}{p^j}\right\rfloor\leq w+\frac{2(p+1)r_0}{p-1}.
\end{equation*}
This implies that 
\begin{multline*}
	F_{k_{w}}^{(s_0)}(r_0-1)-2w\\
\begin{aligned}
=&	\frac{(p-3)(2r_0-3)+\theta_{k_{w},\ell+r_0-1}^{(s_0)}-\theta_{k_{w},\ell+r_0-2}^{(s_0)}}{2}+2\left(\sum\limits_{i=B^{(s_0)}_{k_{w},\ell-r_0+2}+1}^{B^{(s_0)}_{k_{w},\ell+r_0-1}}-\sum\limits_{i=\ell-r_0+2}^{\ell+r_0-2}\right)\v(i)-2w
	\\\leq& \frac{(p-3)(2r_0-3)+p-5}{2}+\frac{4(p+1)r_0}{p-1}-2p^{2p+2}
	\\
	<&-4p-6.
\end{aligned}
\end{multline*}
This proves the right hand side of \eqref{neq::7}.

(2)  By Lemma~\ref{lem:9}(2), we have $$B^{(s_0)}_{k_{w_0},\ell+r_0}<B^{(s)}_{k_{w_0},\ell+r_0}\quad
\textrm{and}\quad B^{(s_0)}_{k_{w_0},\ell-r_0+1}>B^{(s)}_{k_{w_0},\ell-r_0+1}.$$
Combined with \eqref{neq::8}, these equalities show that
\begin{equation*}
	F^{(s)}_{k_{w_0}}(r_0)-F^{(s_0)}_{k_{w_0}}(r_0)=2(\theta_{k_{w_0},r_0+\ell}^{(s)}-\theta_{k_{w_0},r_0+\ell}^{(s_0)})+2\left(\sum\limits_{i=B^{(s)}_{k_{w_0},\ell-r_0+1}+1}^{B^{(s_0)}_{k_{w_0},\ell-r_0+1}}+\sum\limits_{i=B^{(s_0)}_{k_{w_0},\ell+r_0}+1}^{B^{(s)}_{k_{w_0},\ell+r_0}}\right)\v(i)>0.
\end{equation*}
Combined with (1), this proves that $F^{(s)}_{k_{w_0}}(r_0)>0$.
Similarly, we have
\begin{equation*}
	F^{(s)}_{k_{w_0}}(r_0-1)<F^{(s_0)}_{k_{w_0}}(r_0-1)\leq 0.\qedhere
\end{equation*}
\end{proof}
\begin{convention}
We use $w_0$ to denote the unique integer determined in Lemma~\ref{mlem:1}.
\end{convention}

\begin{lemma}\label{aa:1}
For every $s\in S$ and $1\leq r\leq r_0+2$, we have
$$\Delta^{(s)}_{k_{w_0},\ell+r}-\Delta^{(s)}_{k_{w_0},\ell+r-1}=\frac{1}{2}F_{k_{w_0}}^{(s)}(r)+P_{k,\ell}.$$
\end{lemma}
\begin{proof}
	We first prove 	 
	\begin{equation}\label{re::8}
			\Delta^{(s)}_{k, \ell-r+1}-\Delta^{(s)}_{k, \ell-r}+\Delta^{(s)}_{k, \ell+r}-\Delta^{(s)}_{k, \ell+r-1}\\
	=2P_{k,\ell}.
	\end{equation}
	It is enough to check the three hypotheses in  Lemma~\ref{lem:7}. 
	By Lemma~\ref{re:10}, we have
		\begin{equation}\label{aa:5}
			d_{k_{w_0}}^{\new}\left(s\right) =	d_{k_{w_0}}^{\Iw}\left(s\right)-	2d_{k_{w_0}}^{\ur}\left(s\right)\geq 2k_{w_0,\bullet}-4-\frac{4k_{w_0,\bullet}}{p+1}-4\geq 3\ell.
		\end{equation}
	This proves the first hypothesis. 	
The second hypothesis is proved by $1\leq r_0+2\leq p^{\v(\ell)-1}$.
	From our assumption $1\leq r\leq r_0+2$, the maximal $p$-adic valuation of the integers in $(B^{(s)}_{k,\ell-r+1}, B^{(s)}_{k,\ell+r}] $ is $w<\v(\ell)$. This proves the last hypothesis. 
	
	Combining   \eqref{aa:2} with \eqref{re::8}, we complete the proof.
\end{proof}

\begin{lemma}\label{aa:4}
For each $s\in S$, 
$\ell_0,\ell_1,\ell_2\in \{0,\dots, \frac12d_k^{\new}(s)\}$ such that 
$\v(\ell_0)\geq 1$, $\ell_1\in \{\ell_0,\ell_0+1\}$, $\ell_1\leq \ell_2$, if the largest $p$-adic valuation of the integers in $[\ell_0+1,\ell_2]$ is strictly less than $\frac{1}{2}p^{\v(\ell_0)}$, then 
$${\Delta}_{k_{w_0}, \ell_2+1}-{\Delta}_{k_{w_0}, \ell_2}-(	{\Delta}_{k_{w_0}, \ell_1+1}-{\Delta}_{k_{w_0}, \ell_1})>0.$$
\end{lemma}
\begin{proof}
	Adding and subtracting terms of the form ${\Delta}_{k_{w_0}, \ell'+1}-{\Delta}_{k_{w_0}, \ell'}$
	into the left hand side of the last inequality, where $\ell'$ runs over all integers in $[\ell_1+1, \ell_2-1]$ with $\v(\ell')\geq \v(\ell_0)$, by induction, we can assume $\ell_2\leq \ell_0+p^{\v(\ell_0)}$.
	By Corollary~\ref{corollary:r_1r_2}, we have
	$${\Delta}_{k_{w_0}, \ell_2+1}-{\Delta}_{k_{w_0}, \ell_2}-	{\Delta}_{k_{w_0}, \ell_1+1}-{\Delta}_{k_{w_0}, \ell_1}\geq 2(\ell_2-\ell_1)-2	\sum\limits_{j=\ell_1+1}^{\ell_2}\v(j)-1.$$
Hence, it is enough to prove that for $\ell_1\in \{\ell_0,\ell_0+1\}$ and $\ell_1\leq \ell_2\leq \ell_0+p^{\v(\ell_0)}$, we have
$$2(\ell_2-\ell_1)-2	\sum\limits_{j=\ell_1+1}^{\ell_2}\v(j)-1> 0.$$
We prove it by taking induction on $\v(\ell_0)$.

	When $\v(\ell_0)=1$, if $\ell_2<\ell_0+p$, this statement is trivial. If $\ell_2=\ell_0+p$, then we have
$$2(\ell_2-\ell_1)-2	\sum\limits_{j=\ell_1+1}^{\ell_2}\v(j)-1\geq 2p-3-2p^{\frac{1}{2}}>0.
$$
This proves the base case.

Now assume that $\v(\ell_0)\geq 2$. Note that every 
$\ell_0+1\leq j\leq \ell_0+p^{\v(\ell_0)}-1$ satisfies $\v(j)\leq \v(\ell_0)<p^{\frac12\v(\ell_0)}$. Hence, if $\ell_2\neq \ell_0+p^{\v(\ell_0)}$, this statement can be deduce from cutting interval 
$[\ell_1+1, \ell_2]$ 
by integers of $p$-adic valuation $\v(\ell)-1$. By our induction and a similar argument as in the first simplification, we completes the proof of this case. Now, we deal with the case 
$\ell_2=\ell_0+p^{\v(\ell_0)}$. In this case, we have
$$2(\ell_2-\ell_1)-2	\sum\limits_{j=\ell_1+1}^{\ell_2}\v(j)-1\geq 2p^{\v(\ell_0)}-2\times\frac{p^{\v(\ell_0)}-1}{p-1}-3-\v(\ell_2)>0,$$
where the last equality uses our assumption $\v(\ell_2)\leq \frac{1}{2}\v(\ell_0)$.
This completes the proof. 
\end{proof}

\begin{lemma}\label{aa:6}
	For each $s\in S$, 
	$\ell_0,\ell_1,\ell_2\in \{0,\dots, \frac12d_k^{\new}(s)\}$ such that 
	$\v(\ell_0)\geq 1$, $\ell_1\in \{\ell_0-1,\ell_0\}$, $\ell_1\geq \ell_2$, if the largest $p$-adic valuation of the integers in $[\ell_2, \ell_0-1]$ is strictly less than $\frac{1}{2}p^{\v(\ell_0)}$, then 
	$${\Delta}_{k_{w_0}, \ell_2}-{\Delta}_{k_{w_0}, \ell_2-1}-(	{\Delta}_{k_{w_0}, \ell_1}-{\Delta}_{k_{w_0}, \ell_1-1})<0.$$
\end{lemma}
\begin{proof}
	Its proof is exactly same to Lemma~\ref{aa:4}.
\end{proof}

\begin{lemma}\label{mlemma1}
For any $s\in S$, we have
\begin{enumerate}
	\item 	$\Delta^{(s)}_{k_{w_0}, \ell+r}-\Delta^{(s)}_{k_{w_0}, \ell+r-1} 
		\leq	P_{k_{w_0},\ell}$ for $1\leq r\leq r_0-2$;
	\item 
	$\Delta^{(s)}_{k_{w_0}, \ell+r}-\Delta^{(s)}_{k_{w_0}, \ell+r-1} 
	>	P_{k_{w_0},\ell}$ for $r_0+2\leq r\leq \frac{1}{2}d_k^\new(s)-\ell.$
\end{enumerate}

\end{lemma}
\begin{proof}
	For simplicity we put $k:=k_{w_0}$ temporarily in this Lemma.
	
(1) 
By Lemmas~\ref{mlem:1} and \ref{aa:1}, 
we have 
$\Delta^{(s)}_{k, \ell+r_0-1}-\Delta^{(s)}_{k, \ell+r_0-2}\leq  P_{k,\ell}.$
Note that by \eqref{aa:5},  we have $\frac12d_k^{\new}\leq \frac{p-1}{p+1}k_\bullet-4\leq 3\ell$, 
and hence every integer in $\{1,\dots, \frac{1}{2}d_k^{\new}\}$ has $p$-adic valuation less than or equal to $\v(\ell)=p^{p+2}$.
Combined with $p^{\v(\ell+r_0)}=p^{2p}> 2\v(\ell), $ this result checks the conditions in Lemma~\ref{aa:6}, by which, we have
	$$\Delta^{(s)}_{k_{w_0}, \ell+r}-\Delta^{(s)}_{k_{w_0}, \ell+r-1} <
	\Delta^{(s)}_{k_{w_0}, \ell+r_0-1}-\Delta^{(s)}_{k_{w_0}, \ell+r_0-2}\leq P_{k,\ell},$$ 
for every $1\leq r\leq r_0-2$.

%
%
%

	(2) We first prove that 
	\begin{equation}\label{eq::30}
		F^{(s)}_k(r_0+2)
			\geq 	F^{(s)}_k(r_0-1)+2.
	\end{equation}
Note that $\ell+r_0$ and $\ell-r_0$ are the unique $p$-divisible integers in $(\ell+r_0-1-1, \ell+r_0+2-1]$ and $[\ell-r_0-2+1, \ell-(r_0-1)+1)$, respectively. 
	By Lemma~\ref{lem:8}(3) and \eqref{eq::1}, we have
		\begin{multline}\label{mlemma:1}
			F^{(s)}_k(r_0+2)-F^{(s)}_k(r_0-1)\\
			\begin{aligned}
			=&
			\frac{6(p-3)+2(\theta_{\ell+r_0+2}^{(s)}-\theta_{\ell+r_0+1}^{(s)})}{2}-2\v(\ell+r_0)-2\v(\ell-r_0) +2\times\left(\sum\limits_{i=B^{(s)}_{k,\ell+r_0-1}+1}^{B^{(s)}_{k,\ell+r_0+2}}+\sum\limits_{i=B^{(s)}_{k,\ell-r_0-1}+1}^{B^{(s)}_{k,\ell-r_0+2}}\right)\v(i)\\
			\geq &
			3(p-3)-(p-5)-8p+
			2\times\left(\sum\limits_{i=B^{(s)}_{k,\ell+r_0-1}+1}^{B^{(s)}_{k,\ell+r_0+2}}+\sum\limits_{i=B^{(s)}_{k,\ell-r_0-1}+1}^{B^{(s)}_{k,\ell-r_0+2}}\right)\v(i).
	\end{aligned}
		\end{multline}

	By Lemma~\ref{lem:9}(2), we have \begin{gather}\label{aa1}
		{B^{(s)}_{k,\ell\pm r_0+2}}=Q_w\pm \frac{r_0(p+1)}{2}+	s_1+s-k_0+1,\\
		\label{aa2}
		B^{(s)}_{k,\ell\pm  r_0-1}=Q_w\pm \frac{r_0(p+1)}{2}+	s_1-s-p-1.
	\end{gather}
Note that  from $s_1,s\in S$, we have
$s_1+s\geq k_0+1$, and hence 	$$s_1+s-k_0+1> 0.$$
Combined with $s_1-s-p\leq 0$, \eqref{aa1} and \eqref{aa2}, this inequality implies that there is a unique $p$-divisible integer in $[B^{(s)}_{k,\ell+r_0-1}+1,B^{(s)}_{k,\ell+r_0+2}]$ (resp. $[B^{(s)}_{k,\ell-r_0-1}+1,B^{(s)}_{k,\ell-r_0+2}]$) of $p$-adic valuation $2p$, and further that  
	\begin{equation*}
		\sum\limits_{i=B^{(s)}_{k,\ell+r_0-1}+1}^{B^{(s)}_{k,\ell+r_0+2}}\v(i)= 	\sum\limits_{i=B^{(s)}_{k,\ell-r_0-1}+1}^{B^{(s)}_{k,\ell-r_0+2}}\v(i)=2p.
	\end{equation*}
	Plugging it into \eqref{mlemma:1}, we obtain $F^{(s)}_k(r_0+2)-F^{(s)}_k(r_0-1)\geq 2p-4\geq 2$, 
and hence complete the proof of \eqref{eq::30}. 	
	Combining \eqref{eq::30} with Lemma~\ref{mlem:1} that $F^{(s)}_k(r_0-1)\geq -1$, we have
	$	F^{(s)}(r_0+2)> 0.$
Together with Lemma~\ref{aa:1}, 
this implies that
$\Delta^{(s)}_{k, \ell+r_0+2}-\Delta^{(s)}_{k, \ell+r_0-1}>  P_{k,\ell}.$
Similar to the argument in (1), by Lemma~\ref{aa:6}, we have
$$\Delta^{(s)}_{k_{w_0}, \ell+r}-\Delta^{(s)}_{k_{w_0}, \ell+r-1} >
\Delta^{(s)}_{k_{w_0}, \ell+r_0+2}-\Delta^{(s)}_{k_{w_0}, \ell+r_0+1}> P_{k,\ell},$$ 
for every $r_0-2\leq r\leq \frac{1}{2}d_k^\new(s)-\ell$. This completes the proof.
\end{proof}

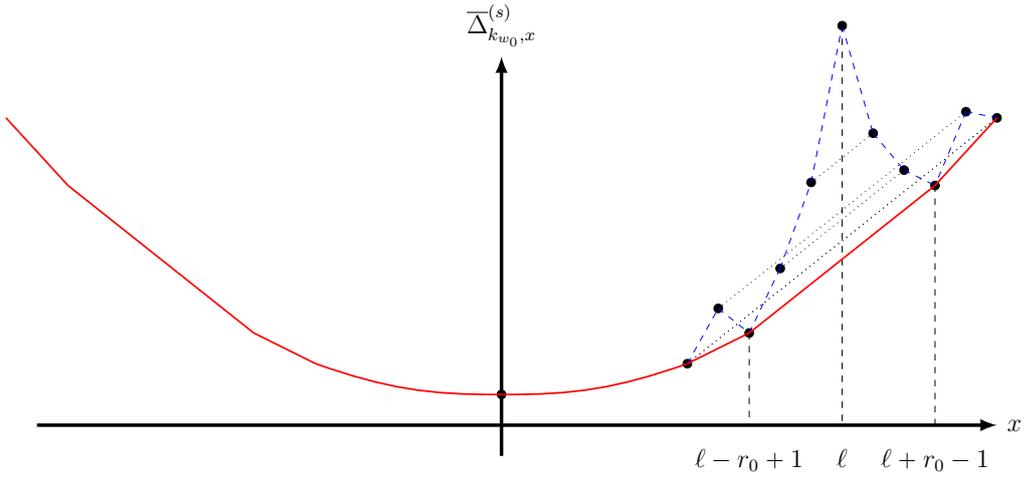
\begin{figure}
	\centering
	  \resizebox{0.8\textwidth}{!}{%
	\begin{tikzpicture}
		\draw[axis] (-7,0) -- (7,0) node[right=2* \nudge cm] {\(x\)};
		\draw[axis] (0,0) -- (0,5) node[above=2*\nudge cm] {\(\OD^{(s)}_{k_{w_0},x}\)};
		
		\coordinate (A) at (0,1/2);
		\coordinate (B) at (3,2/2);
		\coordinate (D) at (3.5,3.8/2);
		\coordinate (E) at (4,3/2);
			\coordinate (E1) at (4.5,5.1/2);
				\coordinate (E2) at (5,7.9/2);
					\coordinate (G1) at  (6,9.5/2);
			\coordinate (G2) at (6.5,8.3/2);
		\coordinate (F) at (7.5, 10.2/2);
		\coordinate (G) at (7, 7.8/2);
		\coordinate (H) at (5.5, 6.5);
				\coordinate (C) at (8,10/2);
		\coordinate (H0) at (5.5, 0);
		\coordinate (E0) at (4,0);
		\coordinate (G0) at (7, 0);
		
		\filldraw[black] (A) circle (2pt) node[anchor=west]{};
		\filldraw[black] (B) circle (2pt) node[anchor=west]{};
		\filldraw[black] (C) circle (2pt) node[anchor=west]{};
		\filldraw[black] (D) circle (2pt) node[anchor=west]{};
		\filldraw[black] (E) circle (2pt) node[anchor=west]{};
		\filldraw[black] (F) circle (2pt) node[anchor=west]{};
		\filldraw[black] (G) circle (2pt) node[anchor=west]{};
		\filldraw[black] (H) circle (2pt) node[anchor=west]{};
		\filldraw[black] (E1) circle (2pt) node[anchor=west]{};
	\filldraw[black] (E2) circle (2pt) node[anchor=west]{};
		\filldraw[black] (G2) circle (2pt) node[anchor=west]{};
		\filldraw[black] (G1) circle (2pt) node[anchor=west]{};

		\draw[line] (A) to [bend left=-10] (B);
		\draw[blue, dashed] (B) to (D);
		\draw[blue, dashed] (D) to (E);
		\draw[blue, dashed] (E) to (E1);
		\draw[blue, dashed] (E1) to (E2);
		\draw[blue, dashed] (E2) to (H);
		\draw[blue, dashed] (H) to (G1);
		\draw[blue, dashed] (G1) to (G2);
		\draw[blue, dashed] (G2) to (G);
		\draw[blue, dashed] (F) to (G);								
		\draw[blue, dashed] (F) to (C);

		\draw[dotted] (B) to (C);
		\draw[dotted] (D) to (F);
		\draw[dotted] (E1) to (G2);
		\draw[dotted] (E2) to (G1);
		\draw[dotted] (B) to (C);

		\draw[dashed] (H) to (H0)  node[below=0.5*\nudge cm] {\(\ell\)};
		\draw[dashed] (E) to (E0)  node[below=0.5*\nudge cm] {\(\ell-r_0+1\)};
		\draw[dashed] (G) to (G0)  node[below=0.5*\nudge cm] {\(\ell+r_0-1\)};

		\coordinate (E1) at (-4,3/2);
		\coordinate (G1) at (-7, 7.8/2);
		\coordinate (B1) at (-3,2/2);
		
		\coordinate (C1) at (-8,10/2);

		\draw[line] (A) to [bend left=10] (B1);

		\draw[line] (B) to (E);
		\draw[line] (E) to (G);
		\draw[line] (C) to (G);	
		\draw[line] (B1) to (E1);
		\draw[line] (E1) to (G1);
		\draw[line] (C1) to (G1);	
		
	\end{tikzpicture}
}
	\caption{The graph of $\OD^{(s)}_{k_{w_0}}$ for $s\leq s_1$.}
	\label{fig1}
\end{figure}

\begin{figure}
	\centering

  \resizebox{0.8\textwidth}{!}{%
\begin{tikzpicture}
	
		\draw[axis] (-7,0) -- (7,0) node[right=2* \nudge cm] {\(x\)};
\draw[axis] (0,0) -- (0,5) node[above=2*\nudge cm] {\(\OD^{(s)}_{k_{w_0},x}\)};

\coordinate (A) at (0,1/2);
\coordinate (B) at (3,2/2);
\coordinate (D) at (3.5,3.8/2);
\coordinate (E) at (4,2);
\coordinate (E1) at (4.5,5.1/2);
\coordinate (E2) at (5,7.9/2);
\coordinate (G1) at  (6,9.5/2);
\coordinate (G2) at (6.5,8.3/2);
\coordinate (F) at (7.5, 10.2/2);
\coordinate (G) at (7, 8.8/2);
\coordinate (H) at (5.5, 6.5);
\coordinate (C) at (8,10/2);
\coordinate (H0) at (5.5, 0);
\coordinate (E0) at (4,0);
\coordinate (G0) at (7, 0);

\filldraw[black] (A) circle (2pt) node[anchor=west]{};
\filldraw[black] (B) circle (2pt) node[anchor=west]{};
\filldraw[black] (C) circle (2pt) node[anchor=west]{};
\filldraw[black] (D) circle (2pt) node[anchor=west]{};
\filldraw[black] (E) circle (2pt) node[anchor=west]{};
\filldraw[black] (F) circle (2pt) node[anchor=west]{};
\filldraw[black] (G) circle (2pt) node[anchor=west]{};
\filldraw[black] (H) circle (2pt) node[anchor=west]{};
\filldraw[black] (E1) circle (2pt) node[anchor=west]{};
\filldraw[black] (E2) circle (2pt) node[anchor=west]{};
\filldraw[black] (G2) circle (2pt) node[anchor=west]{};
\filldraw[black] (G1) circle (2pt) node[anchor=west]{};

\draw[line] (A) to [bend left=-10] (B);
\draw[blue, dashed] (B) to (D);
\draw[blue, dashed] (D) to (E);
\draw[blue, dashed] (E) to (E1);
\draw[blue, dashed] (E1) to (E2);
\draw[blue, dashed] (E2) to (H);
\draw[blue, dashed] (H) to (G1);
\draw[blue, dashed] (G1) to (G2);
\draw[blue, dashed] (G2) to (G);
\draw[blue, dashed] (F) to (G);								
\draw[blue, dashed] (F) to (C);

\draw[dotted] (E) to (G);
\draw[dotted] (D) to (F);
\draw[dotted] (E1) to (G2);
\draw[dotted] (E2) to (G1);
\draw[dotted] (B) to (C);

\draw[dashed] (H) to (H0)  node[below=0.5*\nudge cm] {\(\ell\)};
\draw[dashed] (E) to (E0)  node[below=0.5*\nudge cm] {\(\ell-r_0+1\)};
\draw[dashed] (G) to (G0)  node[below=0.5*\nudge cm] {\(\ell+r_0-1\)};

\coordinate (E1) at (-4,3/2);
\coordinate (G1) at (-7, 7.8/2);
\coordinate (B1) at (-3,2/2);

\coordinate (C1) at (-8,10/2);

\draw[line] (A) to [bend left=10] (B1);

\draw[line] (B) to (C);

\draw[line] (B1) to (C1);
\end{tikzpicture}
}

\caption{The graph of $\OD^{(s)}_{k_{w_0}}$ for $s\geq s_1+1$.}
\label{fig2}
\end{figure}
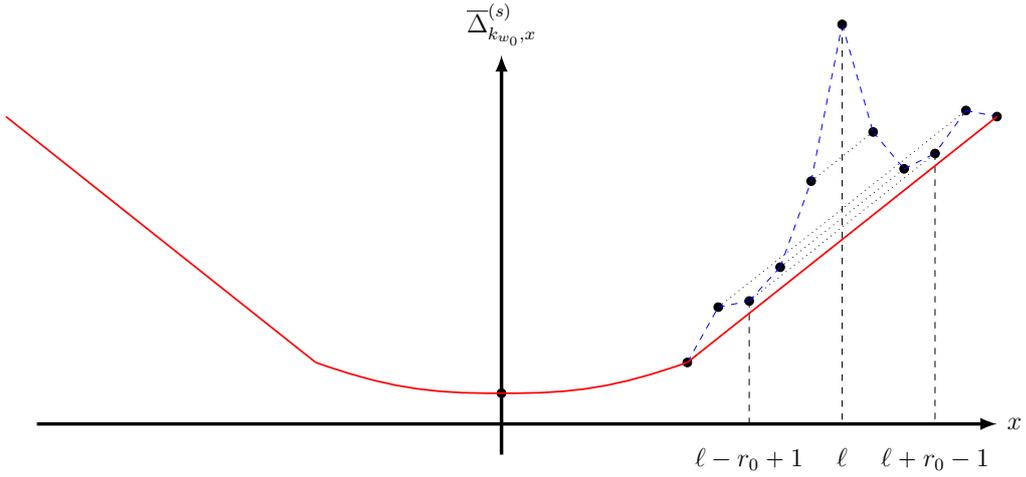

Here, the dotted lines and the long red lines are all of slope $P_{k_{w_0},\ell}$.

\begin{proposition}\label{prop:1}
	For each $s\in S$,
	\begin{enumerate}
		\item if  $s\leq s_1$, then the segment with center $\ell$ in $\OD^{(s)}_{k_{w_0}}$ has length $2(r_0-1)$ and slope $P_{k_{w_0},\ell}$;
		\item if $s\geq s_1+1$, then the segment with center $\ell$ in $\OD^{(s)}_{k_{w_0}}$ has length $2(r_0+1)$ and slope $P_{k_{w_0},\ell}$.
	\end{enumerate}
We demonstrate the two cases in Figures~\ref{fig1} and \ref{fig2}. 
\end{proposition}
\begin{proof}
	For simplicity, we put $k:=k_{w_0}$.
	By \eqref{eq::1} and Lemma~\ref{lem:9}(1), we have
\begin{multline}\label{eq::5}
		F^{(s)}_k(r_0+1)-F^{(s)}_k(r_0)\\
		\begin{aligned}	
=&
		p-3+\theta_{k,\ell+r_0+1}^{(s)}-\theta_{k,\ell+r_0}^{(s)}-2\v(\ell+r_0)-2\v(\ell-r_0)
		+2\left(\sum\limits_{i=B^{(s)}_{k,\ell+r_0}+1}^{B^{(s)}_{k,\ell+r_0+1}}+\sum\limits_{i=B^{(s)}_{k,\ell-r_0}+1}^{B^{(s)}_{k,\ell-r_0+1}}\right)\v(i)\\
		=&	p-3-p-1+2(p-1-2s+k_0)-8p+2\left(\sum\limits_{i=B^{(s)}_{k,\ell+r_0}+1}^{B^{(s)}_{k,\ell+r_0+1}}+\sum\limits_{i=B^{(s)}_{k,\ell-r_0}+1}^{B^{(s)}_{k,\ell-r_0+1}}\right)\v(i)
		\\
		=&
		-4+2(k_0-1-2s)-6p+2\left(\sum\limits_{i=B^{(s)}_{k,\ell+r_0}+1}^{B^{(s)}_{k,\ell+r_0+1}}+\sum\limits_{i=B^{(s)}_{k,\ell-r_0}+1}^{B^{(s)}_{k,\ell-r_0+1}}\right)\v(i).
	\end{aligned}
\end{multline}
	
(1)		By Lemma~\ref{lem:9}(2), we have $${B^{(s)}_{k,\ell+r_0+1}}=Q_w+ \frac{r_0(p+1)}{2}+s_1-s\quad\textrm{and}\quad B^{(s)}_{k,\ell+r_0}=Q_w+\frac{r_0(p+1)}{2}+s_1+s-k_0-p.$$
Note that from $s,s_1\in S$, we have 
\begin{equation}\label{aa:7}
	s+s_1\leq2\times \left\lfloor\frac{k_0-4+p}{2}\right\rfloor \leq k_0-3+p.
\end{equation}
	Combined with our assumption $s\leq s_1$, this chain of inequalities implies that there is a unique $p$-divisible integer in $[B^{(s)}_{k,\ell+r_0}+1,B^{(s)}_{k,\ell+r_0+1}]$ of $p$-adic valuation $2p$, and further that
\begin{equation}\label{neq::10}		\sum\limits_{i=B^{(s)}_{k,\ell+r_0}+1}^{B^{(s)}_{k,\ell+r_0+1}}\v(i)= 2p.
\end{equation}
	
Similarly, we have
$${B^{(s)}_{k,\ell-r_0+1}}=Q_w- \frac{r_0(p+1)}{2}+s_1-s\quad\textrm{and}\quad B^{(s)}_{k,\ell-r_0}=Q_w-\frac{r_0(p+1)}{2}+s_1+s-k_0-p,$$
and hence \begin{equation}\label{neq::11}
	\sum\limits_{i=B^{(s)}_{k,\ell-r_0}+1}^{B^{(s)}_{k,\ell-r_0+1}}\v(i)= 2p.
\end{equation}
				
	Plugging \eqref{neq::10}  and \eqref{neq::11} into \eqref{eq::5}, we have \begin{equation}\label{m2}
		F^{(s)}_k(r_0+1)-F^{(s)}_k(r_0)=2(p-2+k_0-1-2s)\geq 0,
	\end{equation}
where the last inequality is same to \eqref{aa:7}.
	Combining \eqref{m2} with Lemma~\ref{aa:1} and our assumption $F^{(s)}_k(r_0)> 0$,
we have
	\begin{equation*}
		\Delta^{(s)}_{k, \ell+r_0+1}-\Delta^{(s)}_{k, \ell+r_0}
			> P_{k,\ell}  \quad\textrm{and}\quad\Delta^{(s)}_{k, \ell+r_0}-\Delta^{(s)}_{k, \ell+r_0-1}			> P_{k,\ell}.
	\end{equation*}
	 Further combining the above two inequalities with  Corollary~\ref{corollary:slope} and Lemma~\ref{mlemma1}, we complete the  proof of the first statement in this proposition. 
	
(2) Note that by Lemma~\ref{mlem:1}, we have
$F^{(s)}_k(r_0)\geq 0$. Combined with Lemma~\ref{aa:1}, this implies that
\begin{equation}\label{eq::6}
	\Delta^{(s)}_{k, \ell+r_0}-\Delta^{(s)}_{k, \ell+r_0-1}\geq P_{k,\ell}.
\end{equation}
Consider
	\begin{multline}\label{m1}
		F^{(s)}_k(r_0+1)-F^{(s)}_k(r_0-1)\\
		\begin{aligned}
			=
			2(p-3)-2\v(\ell+r_0)-2\v(\ell-r_0)-2\v(\ell+r_0-1)-2\v(\ell-r_0+1)&\\
			+2\left(\sum\limits_{i=B^{(s)}_{k,\ell+r_0-1}+1}^{B^{(s)}_{k,\ell+r_0+1}}+\sum\limits_{i=B^{(s)}_{k,\ell-r_0}+1}^{B^{(s)}_{k,\ell-r_0+2}}\right)\v(i)&.
		\end{aligned}
	\end{multline}
	By Lemma~\ref{lem:9}(2), we have 
	$$B^{(s)}_{k,\ell+r_0+1}=Q_w+ \frac{r_0(p+1)}{2}+s_1-s\quad\textrm{and}\quad B^{(s)}_{k,\ell+r_0-1}=B^{(s)}_{k,\ell+r_0+1}-(p+1).$$
	From our assumption $s\geq s_1+1$, we have
\begin{equation}\label{re::9}
		\sum\limits_{i=B^{(s)}_{k,\ell+r_0-1}}^{B^{(s)}_{k,\ell+r_0+1}}\v(i)=1.
\end{equation}
	By Lemma~\ref{lem:9}(2), we have 
$$B^{(s)}_{k,\ell-r_0+2}=Q_w- \frac{r_0(p+1)}{2}+s_1+s-k_0+1\quad\textrm{and}\quad B^{(s)}_{k,\ell-r_0}=B^{(s)}_{k,\ell-r_0+2}-(p+1).$$
Note that $k_0\leq s+s_1\leq k_0+p-3$. Hence, we have
$$	\sum\limits_{i=B^{(s)}_{k,\ell-r_0}}^{B^{(s)}_{k,\ell-r_0+2}}\v(i)=2p.
	$$
	Plugging this equality and \eqref{re::9} into \eqref{m1}, we have
\begin{equation}\label{eq::14}
	F^{(s)}_k(r_0+1)- F^{(s)}_k(r_0-1)=2(p-3)-4p+2=-2p-4.
\end{equation}

	Similar to the above argument, we have $$\sum\limits_{i=B^{(s)}_{k,\ell+r_0}+1}^{B^{(s)}_{k,\ell+r_0+1}}\v(i)=\sum\limits_{i=B^{(s)}_{k,\ell-r_0}+1}^{B^{(s)}_{k,\ell-r_0+1}}\v(i)=0.$$
	
Combined with Lemma~\ref{lem:9}(1), this chain of equalities implies that
	\begin{align*}
			&F^{(s)}_k(r_0)-F^{(s)}_k(r_0-1)	\\
			\notag=&
			p-3+\theta_{k,\ell+r_0}^{(s)}-\theta_{k,\ell+r_0-1}^{(s)}-2\v(\ell+r_0-1)-2\v(\ell-r_0+1)
		+\left(\sum\limits_{i=B^{(s)}_{k,\ell+r_0-1}+1}^{B^{(s)}_{k,\ell+r_0}}+\sum\limits_{i=B^{(s)}_{k,\ell-r_0+1}+1}^{B^{(s)}_{k,\ell-r_0}}\right)\v(i)
			\\
		\notag	=&	2(2s+2-k_0)-4\\
		\notag	=&	2(2s-k_0).
	\end{align*}
Together with \eqref{eq::14}, this chain of equalities implies that
$$F^{(s)}_k(r_0+1)+F^{(s)}_k(r_0)\leq 
2F^{(s)}_k(r_0-1)-2p-4+2(2s-k_0+2)-4<0.$$
Together with Lemma~\ref{aa:1}, this inequality implies that
	\begin{equation*}
	\Delta^{(s)}_{k, \ell+r_0+1}-\Delta^{(s)}_{k, \ell+r_0-1}<2P_{k,\ell}.
\end{equation*}
This inequality with 
\eqref{eq::6}, Corollary~\ref{corollary:slope} and Lemma~\ref{mlemma1}(1)  shows that 
for any $1\leq r\leq r_0+1$, the point $(n,\Delta_{k,\ell+r}^{(s)})$ is above the  segment connecting $(\ell-r_0-1, \Delta^{(s)}_{k,\ell-r_0-1})$ and  $(\ell+r_0+1, \Delta^{(s)}_{k,\ell+r_0+1})$. Combined with Lemma~\ref{mlemma1}(2), this completes the proof.
\end{proof} 

%
%

We introduce some notations from \cite[Definition~5.11]{xiao}.

\begin{notation}
	
	\begin{enumerate}
		\item 	For any $k\in \calK$, $s\in \{0,\dots,p-2\}$ and $w_\star\in \bfm_{\CC_p}$, we denote by $L^{(s)}_{w_{\star}, k}$ the largest number (if exists) in $\left\{1, \ldots, \frac{1}{2} d_{k}^{\text {new }}(s)\right\}$ such that
		$$
		v_{p}\left(w_{\star}-w_{k}\right) \geq \OD^{(s)}_{k, L^{(s)}_{w_{\star}, k}}-\OD^{(s)}_{k, L^{(s)}_{w_{\star}, k}-1}.
		$$
		\item	When such $L_{w_{\star}, k}$ exists, the open interval $n S_{w_{\star}, k}^{(s),\dagger}:=\Big(\frac{1}{2} d_{k}^{\Iw,\dagger}(s)-L^{(s)}_{w_{\star}, k}, \frac{1}{2} d_{k}^{\Iw,\dagger}(s)+L^{(s)}_{w_{\star}, k}\Big)$ is called the \emph{near-Steinberg range} for the pair $\left(w_{\star}, k\right)$. 
		
		\item
		Note that by Lemma~\ref{re:15}, we have proved that $\iota(s)$ is generic for every $s\in S$.
		Hence, by \cite[Proposition~5.17(1)]{xiao}, for fixed $w_\star\in \CC_p$ and $s\in S,$ 
		the set of near-Steinberg ranges $\nS^{(s),\dagger}_{w_{\star}, k}$ for all $k\in \calK$ is nested, i.e., for any two such open intervals, either they are disjoint or one is contained in another. We call those $\mathrm{n} S^{(s),\dagger}_{w_{\star}, k}$ that are not contained in the others the \emph{maximal near-Steinberg range}.
	\end{enumerate}
	
\end{notation}

\begin{proof}[Proof of Proposition~\ref{main thm}]
	Without loss of generality, we may assume that two of the different indices in $\us$ are $s_1$ and $s_2$ taken in the beginning of this section. 
	We write $k:=k_{w_0}$ and take $w_\star\in \bfm_{\CC_p}$ such that 
	$\v(w_\star-w_k)=P_{k,\ell}$. 
	By Proposition~\ref{prop:1}, we have
\begin{equation}\label{neq::18}
		L^{(s_1)}_{w_\star, k}=\ell+r_0-1\quad\textrm{and}\quad L^{(s_2)}_{w_\star, k}=\ell+r_0+1.
\end{equation}

Now we claim the following.
	\begin{claim}\label{claim}
	The near-Steinberg ranges	$\nS^{(s_i),\dagger}_{w_\star,k}$ for  $i=1,2$ are both maximal. 
	\end{claim} 
	\begin{claim}\label{claim:2}
			For every $k'(\neq k)\in \calK$ such that  $\v(w_{k}-w_{k'})\geq\v(w_{k}-w_{\star})$, 
$$		\textrm{if~}
		\nS^{(s_1),\dagger}_{w_\star,k}\cap [d_k^{\ur,\dagger}(s_1), d_k^{\Iw,\dagger}(s_1)-d_k^{\ur,\dagger}(s_1)]\neq\emptyset,
		\textrm{then~}  \nS^{(s_1),\dagger}_{w_\star,k}\subset  \calK_+^{(s_1)}\textrm{~or~}K^{(s_1)}_-,$$
		where 
		$\calK_+^{(s_i)}$ (resp. $\calK_-^{(s_i)}$) for $i=1,2$ is the set consisting of $k'(\neq k)\in \calK$ such that  $\v(w_{k}-w_{k'})\geq\v(w_{k}-w_{\star})$,
		and $$\nS^{(s_i),\dagger}_{w_\star,k}\subset  \left[d_{k'}^{\ur,\dagger}(s_i), \frac12d_{k'}^{\Iw,\dagger}(s_i)\right] ~(\textrm{resp.~} \left[\frac12d_{k'}^{\Iw,\dagger}(s_i), d_{k'}^{\Iw,\dagger}(s_i)-d_{k'}^{\ur,\dagger}(s_i)\right]). $$
		
	The result holds with $s_{1}$ replaced by $s_{2}$.
	\end{claim}
The proof of Claim~\ref{claim}  is given at the end of this section, while Claim~\ref{claim:2} follows from  \cite[Proposition~5.15(1)]{xiao}.
	Combining the relations $
	d_k^{\Iw,\dagger}(s_1)=d_k^{\Iw,\dagger}(s_2)$, $|d_k^{\ur,\dagger}(s_1)-d_k^{\ur,\dagger}(s_2)|\leq 1,$
\eqref{neq::18} with Claim~\ref{claim:2}, we have 
$\calK_+^{(s_1)}=\calK_+^{(s_2)}$ and $\calK_-^{(s_1)}=\calK_-^{(s_2)}$, which are denoted by $\calK_\pm$ uniformly.

	We next prove that for each $i=1,2$, the Newton polygon $\NP(G^{(s_i),\dagger}(w_\star,-))$ has a segment $\LL_i$ over some interval of length  $\nS^{(s_i),\dagger}_{w_\star,k}$.
	Suppose not. 
	Without loss of generality, we assume that 
	the left endpoint $P$ of $\nS^{(s_i),\dagger}_{w_\star,k}$  is not a vertex.
	By \cite[Thoerem~5.18(2)]{xiao},
	there exists some $k'\in \calK$ such that  $P\in \nS^{(s_i),\dagger}_{w_\star,k'}.$
	Since the near-Steinberg sets are nested, we have 
	$$\nS^{(s_i),\dagger}_{w_\star,k}\subseteq\nS^{(s_i),\dagger}_{w_\star,k'},$$
	which contradicts to our Claim~\ref{claim}.
	This implies that  for $i=1,2,$ we have
	 \begin{align}
	 	&	h^{(s_i),\dagger}_{\frac{1}{2}d_k^{\Iw,\dagger}+L^{(s_i)}_{w_\star, k}}(w_\star)
	 	-h^{(s_i),\dagger}_{\frac{1}{2}d_k^{\Iw,\dagger}-L^{(s_i)}_{w_\star, k}}(w_\star)
	\notag	\\
		=&\v\left(g^{(s_i),\dagger}_{\frac{1}{2}d_k^{\Iw,\dagger}+L^{(s_i)}_{w_\star, k}}(w_\star)\right)
		-\v\left(g^{(s_i),\dagger}_{\frac{1}{2}d_k^{\Iw,\dagger}-L^{(s_i)}_{w_\star, k}}(w_\star)\right)\label{neq::21}\\
		=&\v\left(g^{(s_i)}_{\frac{1}{2}d_k^{\Iw}(s_i)+L^{(s_i)}_{w_\star, k}}(w_\star)\right)
		-\v\left(g^{(s_i)}_{\frac{1}{2}d_k^{\Iw}(s_i)-L^{(s_i)}_{w_\star, k}}(w_\star)\right).\notag
	\end{align}

On the other hand, by Claim~\ref{claim:2}, we have
\begin{align*}
&	\v\left(g^{(s_i)}_{\frac{1}{2}d_k^{\Iw}(s_i)+L^{(s_i)}_{w_\star, k}}(w_\star)\right)
-\v\left(g^{(s_i)}_{\frac{1}{2}d_k^{\Iw}(s_i)-L^{(s_i)}_{w_\star, k}}(w_\star)\right)\\
=&	\v\left(g^{(s_i)}_{\frac{1}{2}d_k^{\Iw}(s_i)+L^{(s_i)}_{w_\star, k},\hat{k}}(w_k)\right)
-\v\left(g^{(s_i)}_{\frac{1}{2}d_k^{\Iw}(s_i)-L^{(s_i)}_{w_\star, k},\hat{k}}(w_k)\right)\\
+&2L^{(s_i)}_{w_\star, k}\times \left(\sum\limits_{k'\in \calK_+}(\v(w_{k'}-w_{\star})-\v(w_{k'}-w_{k}))-\sum\limits_{k'\in \calK_-}(\v(w_{k'}-w_{\star})-\v(w_{k'}-w_{k}))\right)\\
=&2L^{(s_i)}_{w_\star, k}\times\left(\frac{k-2}{2}+\sum\limits_{k'\in \calK_+}(\v(w_{k'}-w_{\star})-\v(w_{k'}-w_{k}))-\sum\limits_{k'\in \calK_-}(\v(w_{k'}-w_{\star})-\v(w_{k'}-w_{k}))\right),
\end{align*}
where the last equality follows from \eqref{neq::20} and \eqref{neq::19}.
Combined  with \eqref{neq::21}, this chain of equality implies that
$$\alpha:=\frac{1}{2L^{(s_i)}_{w_\star, k}}\left(h^{(s_i),\dagger}_{\frac{1}{2}d_k^{\Iw,\dagger}+L^{(s_i)}_{w_\star, k}}(w_\star)
-h^{(s_i),\dagger}_{\frac{1}{2}d_k^{\Iw,\dagger}-L^{(s_i)}_{w_\star, k}}(w_\star)\right)$$ is same for $i=1,2$. Namely, the slopes of the segments $\LL_1$ and $\LL_2$ are both equal to $\alpha$. This implies that
\begin{gather*}
		h^{(s_2),\dagger}_{\frac{1}{2}d_k^{\Iw,\dagger}-L^{(s_1)}_{w_\star, k}+1}(w_\star)
	-h^{(s_2),\dagger}_{\frac{1}{2}d_k^{\Iw,\dagger}-L^{(s_1)}_{w_\star, k}}(w_\star)
	=\alpha<
	h^{(s_1),\dagger}_{\frac{1}{2}d_k^{\Iw,\dagger}-L^{(s_1)}_{w_\star, k}+1}(w_\star)
	-h^{(s_1),\dagger}_{\frac{1}{2}d_k^{\Iw,\dagger}-L^{(s_1)}_{w_\star, k}}(w_\star),\\
	h^{(s_2),\dagger}_{\frac{1}{2}d_k^{\Iw,\dagger}-L^{(s_1)}_{w_\star, k}+2}(w_\star)
	-h^{(s_2),\dagger}_{\frac{1}{2}d_k^{\Iw,\dagger}-L^{(s_1)}_{w_\star, k}+1}(w_\star)=\alpha<h^{(s_1),\dagger}_{\frac{1}{2}d_k^{\Iw,\dagger}-L^{(s_1)}_{w_\star, k}+2}(w_\star)
-h^{(s_1),\dagger}_{\frac{1}{2}d_k^{\Iw,\dagger}-L^{(s_1)}_{w_\star, k}+1}(w_\star).
\end{gather*}
	Combining these two inequalities with  Proposition~\ref{critirien}, we complete the proof.
\end{proof}

\begin{proof}[Proof of Claim~\ref{claim}]
	Suppose that $\nS^{(s_i),\dagger}_{w_\star,k}$ is not maximal for some $i\in \{1,2\}$. Then there is $k'\in \calK$ such that $\nS^{(s_i),\dagger}_{w_\star,k}\subsetneqq \nS^{(s_i),\dagger}_{w_\star,k'}$. By \cite[Proposition~5.15(1)]{xiao}, we have 
$$\v(w_{k}-w_{k'})<\OD^{(s_i)}_{k',L^{(s_i)}_{w_\star,k'}}-\OD^{(s_i)}_{k',L^{(s_i)}_{w_\star,k'}-1}\leq \v(w_{k'}-w_{\star}),$$
and hence $\v(w_{k}-w_{k'})=\v(w_{k}-w_{\star})=P_{k,\ell}$.
Note that \begin{align*}
	2P_{k,\ell}=& 	\frac{(p-1)(2\ell-1)+p+1}{2}
	+\frac{p+1+2(\Dig(\ell)+\Dig(\ell-1))}{p-1}\\
	-&
	\frac{\Dig\left(B^{(s_{k,\ell})}_{k,\ell+1}\right)-\Dig\left(A^{(s_{k,\ell})}_{k,\ell}\right)+\Dig\left(B^{(s_{k,\ell})}_{k,\ell}\right)-\Dig\left(A^{(s_{k,\ell})}_{k,\ell+1}\right)}{p-1}\\
	\geq & (p-1)\ell -
	\frac{B^{(s_{k,\ell})}_{k,\ell+1}-A^{(s_{k,\ell})}_{k,\ell}+B^{(s_{k,\ell})}_{k,\ell}-A^{(s_{k,\ell})}_{k,\ell+1}}{p-1}\\
	=& (p-1)\ell -
	\frac{2(p+1)\ell}{p-1}.
\end{align*}
Since we assume that $p\geq 7$, we have
$P_{k,\ell}\geq \ell$, 
and hence $k'\geq p^{\ell}$.
This implies that 
$$d^{\ur,\dagger}_{k'}(s)>\frac{k'_\bullet}{p+1}\geq \frac{p^{\ell}}{(p+1)(p-1)}>2\ell\geq d^{\ur,\dagger}_{k}(s),$$
and hence  $$d^{\ur,\dagger}_{k}(s)\notin [d^{\ur,\dagger}_{k'}(s),d^{\ur,\dagger}_{k'}(s)-d^{\ur,\dagger}_{k'}(s)],$$
a contradiction.  This completes the proof of our claim.
\end{proof}


\begin{thebibliography}{9999}	
			\bibitem[BP-1]{BP1}
	J. Bergdall and R. Pollack, 
	Slopes of modular forms and the ghost conjecture, {\it Int. Math. Res. Not.} {\bf 2019}(2019), no. 4, 1125--1144.
	
	\bibitem[BP-2]{BP2}
	J. Bergdall and R. Pollack, 
	Slopes of modular forms and the ghost conjecture II, {\it Trans. Amer. Math. Soc.} {\bf 374}(2019), no. 4, 357--388.
	
		\bibitem[Buz]{Bu} K. Buzzard, Eigenvarieties, in L-functions and Galois representations, {\it London Math. Soc. Lecture Note Ser.} {\bf 320}(2007), 59--120.
	
	\bibitem[BC]{BC} K. Buzzard and F. Calegari, A counterexample to the Gouvêa-Mazur conjecture, {\it Comptes Rendus Math.} {\bf 338}(2004), 751--753.
	
	
	
	\bibitem[CEGGPS]{ceggps} 
A. Caraiani, M. Emerton, T. Gee, D. Geraghty, V. Paskunas and S.-W. Shin, Patching and the $p$-adic local Langlands correspondence, {\it Camb. J. of Math.} {\bf 4}(2016), 197--287.
		

	\bibitem[Cla]{Cl} L. Clay, Some Conjectures About the Slopes of Modular Forms, {\it PhD thesis, Northwestern University}, June 2005.
	
	
	\bibitem[Col]{cole} R. F. Coleman, $p$-adic Banach spaces and families of modular forms, {\it Invent. Math.} {\bf 127}(1997), no. 3, 417--479.
	
	
		\bibitem[GM]{GM92} F. Gouvêa and B. Mazur, Families of modular eigenforms, {\it Math. Comp.} {\bf 58}(1992), 793--805.
		
		\bibitem[Gou]{Gou}
		F. Gouvêa, Where the slopes are, {\it J. Ramanujan Math. Soc.} {\bf 16}(2001), 75--99.
		
		
			
		\bibitem[Kis]{Kis} M. Kisin, The Fontaine-Mazur conjecture for $\mathrm{GL}_{2}$, {\it J. Amer. Math. Soc.} {\bf 22}(2009), no. 3, 641--690.
		
	
		
		\bibitem[LTXZ-1]{xiao}
		R. Liu, N. Truong, L. Xiao and B. Zhao,
		Slope of modular forms and geometry of eigencurves, {\it submitted}.
		
		\bibitem[LTXZ-2]{xiao1}
		R. Liu, N. Truong, L. Xiao and B. Zhao, A local analogue of the ghost conjecture of Bergdall-Pollack, {\it submitted}.

	\bibitem[Loe]{Lo07} D. Loeffler, Spectral expansions of overconvergent modular functions, {\it Int. Math. Res. Not.} {\bf 16}(2007).

%
		\bibitem[Ren]{R} R. Ren, Local Gouvêa--Mazur conjecture, {\it submitted}.

	\end{thebibliography}
\end{document}